\documentclass{article}

\usepackage[margin=2cm]{geometry}
\usepackage{amsmath,amssymb}
\usepackage{amsthm}
\usepackage{mathrsfs,MnSymbol}
\usepackage{cite}
\usepackage{enumerate}
\usepackage{enumitem}
\usepackage{url}
\usepackage[colorlinks]{hyperref}
\usepackage{algorithm}
\usepackage{algpseudocode}
\usepackage[dvipsnames]{xcolor}
\usepackage{color}
\usepackage{listings}
\usepackage{tikz}
\tikzset{
every node/.style={circle, draw, inner sep=2pt},
every picture/.style={thick}
}
\usetikzlibrary{matrix,arrows,calc}

\newtheorem{theorem}{Theorem}
\newtheorem{lemma}[theorem]{Lemma}
\newtheorem{proposition}[theorem]{Proposition}
\newtheorem{corollary}[theorem]{Corollary}

\theoremstyle{definition}
\newtheorem{definition}[theorem]{Definition}
\newtheorem{observation}[theorem]{Observation}
\newtheorem{remark}[theorem]{Remark}
\newtheorem{example}[theorem]{Example}

\newtheorem{question}[theorem]{Question}
\newtheorem{problem}[theorem]{Problem}

\newenvironment{thm}{\begin{theorem}}{\end{theorem}}

\newenvironment{obs}{\begin{observation}\bgroup\rm }{\egroup\end{observation}}

\newcommand{\Gaa}{\sf P_5}
\newcommand{\Gab}{\sf G_2}
\newcommand{\Gac}{\sf G_3}
\newcommand{\Gad}{\sf G_{4}}
\newcommand{\Gae}{\sf G_{5}}
\newcommand{\Gaf}{\sf G_{6}}
\newcommand{\Gag}{\sf G_{7}}
\newcommand{\Gah}{\sf G_8}
\newcommand{\Gai}{\sf G_9}
\newcommand{\Gaj}{\sf G_{10}}
\newcommand{\Gak}{\sf G_{11}}
\newcommand{\Gal}{\sf G_{12}}
\newcommand{\Gam}{\sf G_{13}}
\newcommand{\Gan}{\sf G_{14}}
\newcommand{\Gao}{\sf G_{15}}
\newcommand{\Gap}{\sf G_{16}}
\newcommand{\Gaq}{\sf G_{17}}
\newcommand{\Gar}{\sf G_{18}}
\newcommand{\Gas}{\sf G_{19}}
\newcommand{\Gat}{\sf G_{20}}
\newcommand{\Gau}{\sf G_{21}}
\newcommand{\Gav}{\sf G_{22}}
\newcommand{\Gaw}{\sf G_{23}}
\newcommand{\Gax}{\sf G_{24}}
\newcommand{\Gay}{\sf G_{25}}

\newcommand{\mz}{\operatorname{mz}}

\newcommand{\Forbmz}[1]{\mathbf{Forb}_{\mz\leq #1}}

\title{The forbidden structure for zero forcing number}

\author{
Carlos A. Alfaro\footnote{Banco de M\'exico, M\'exico City, M\'exico, \texttt{alfaromontufar@gmail.com}} \and
Michael D. Barrus \footnote{Department of Mathematics, University of Rhode, Island Kingston, RI 02881, USA, \texttt{barrus@uri.edu}} \and
Sergio Gerardo G\'omez-Galicia\footnote{ Departamento de Matemáticas, CINVESTAV, M\'exico City, M\'exico, \texttt{sgomez@math.cinvestav.mx}} \and
Teresa I. Hoekstra-Mendoza\footnote{Instituto de Matem\'aticas, UNAM, Unidad Juriquilla, Quer\'etaro, M\'exico, \texttt{allizdog01@gmail.com}} \and
Miguel Licona\footnote{Escuela Nacional de Estudios Superiores, UNAM, Juriquilla, Quer\'etaro, M\'exico, \texttt{eliconav23@xanum.uam.mx} } \and
Jephian C.-H. Lin \footnote{Department of Applied Mathematics, National Yang Ming Chiao Tung University, Hsinchu, Taiwan, \texttt{jephianlin@gmail.com}} \and
Juan Pablo Serrano\footnote{Departamento de Matemáticas, CINVESTAV, M\'exico City, M\'exico, \texttt{jpserranop@math.cinvestav.mx}} \and 
Ralihe R. Villagr\'an \footnote{Department of Mathematical Sciences, Worcester Polytechnic Institute, Worcester, USA, \texttt{ralihemath@gmail.com}} }

\date{}

\begin{document}
\maketitle

\begin{abstract}
%To be added later... but it is something like:
%Let $G$ be a graph with $n$ vertices.
%The zero forcing number $Z(G)$ of $G$ is a well studied parameter.
%We are interested in the parameter $\mz(G)$ that is equal to $n-Z(G)$.
%Since $\mz(G)$ is monotone on induced subgraphs, then, for any $k\geq 1$, there exist graphs that are forbidden for the family of graphs with $\mz\leq k$.
%We prove that the minimal forbidden graphs for the graphs with $\mz\leq k$ is finite for any $k\geq 1$.
%And {\color{red} (in progress)} give a characterization of the graphs with $\mz\leq 3$.

The {\it zero forcing number} of a graph $G$, $Z(G)$, is a well-studied parameter which arises from a color changing process and has strong connections to {\it minimum rank}, {\it critical ideals} and related invariants. 
In this work, we consider the complementary parameter $\mz(G) = |V(G)| - Z(G)$. 
This parameter is monotone under taking induced subgraphs. 
This leads us to the study of graphs for which $\mz(G)$ is bounded, via forbidden induced subgraphs.
We prove that the number of minimal forbidden graphs for graphs with $\mz(G)\leq k$ is finite for any $k\geq 1$. 
We determine the complete set of minimal forbidden graphs for the case $k = 3$, and we provide partial characterizations of graphs with $\mz(G) \leq 3$, based on girth. 
Our results suggest new directions for the structural understanding of zero forcing-type parameters.

%{\color{blue}We might try ``Inventiones mathematicae'', ``Journal of Inequalities and Applications'', ``Bulletin of Mathematical Sciences'' or ``Networks''}
\end{abstract}

This paper is dedicated to the memory of our friend Michael D. Barrus. 

\section{Introduction}

The \emph{zero forcing game} is a color-change game where vertices can be colored either blue or white.
At the beginning, the player can pick a set of vertices $B$ and color them blue while others remain white.
The goal is to color all vertices blue through repeated applications of the \emph{color change rule}: If $x$ is a blue vertex and $y$ is the only white neighbor of $x$, then $y$ turns blue, denoted as $x\rightarrow y$.
Such coloring process will be known as \emph{zero forcing process}.
%The \emph{chronological list} of a zero forcing game records the forces $x_i\rightarrow y_i$ in the order of performance.
An initial set of blue vertices $B$ is called a \emph{zero forcing set} if starting with $B$ one can make all vertices blue after a zero forcing process.
The \emph{zero forcing number}, denoted $Z(G)$, is the minimum cardinality of a zero forcing set.

In this work, we study the complementary parameter $\mz(G)=|V(G)|-Z(G)$.
Unlike $Z(G)$, it is known that $\mz(G)$ is monotone on induced subgraphs, that is, if $H$ is an induced subgraph of $G$, then $\mz(H)\leq \mz(G)$ (see \cite{AIMZmr}). 
Then, it is natural to ask for classifications of graphs where these parameters are bounded from above.
Moreover, it is known that $\mz(G)$ is a lower bound for certain graph parameters such as the minimum rank \cite{AIMZmr,cancun} and the algebraic co-rank \cite{alfarolin}.
Therefore, the classification of graphs with $\mz(G)$ at most $k$ might shed some light on the classifications of graphs where these parameters are bounded.

Given a family of graphs $\mathfrak{F}$, a graph $G$ is called $\mathfrak{F}$-{\it free} if no induced subgraph of $G$ is isomorphic to a member of $\mathfrak{F}$. We say a graph $H$ is \textit{forbidden} for graphs with $\mz(G)\leq k$, when $\mz(H)\geq k+1$. We have the following result.

\begin{lemma}
The path $P_k$ with $k$ vertices is forbidden for graphs with $\mz(G)\leq k-2$.
\end{lemma}
\begin{proof}
It follows since $Z(P_k)=1$, and then $\mz(P_k)=k-1$.
\end{proof}

Previously, it was noticed that, among connected graphs, complete graphs are the unique graphs with $\mz(G)=1$.

\begin{thm}
\label{lemma:mzandmrandgamma=1}
    Let $G$ be a connected graph.
    Then, the following are equivalent:
    \begin{enumerate}[label={\rm (\arabic*)}]
        \item $G$ is a complete graph,
        \item $G$ is $P_3$-free,
        \item $\mz(G)\leq 1$.
    \end{enumerate}
\end{thm}

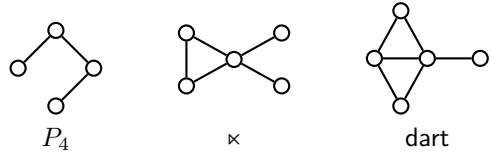
\begin{figure}[h]
\begin{center}
\begin{tabular}{c@{\extracolsep{10mm}}c@{\extracolsep{10mm}}c@{\extracolsep{10mm}}c@{\extracolsep{10mm}}c}
\begin{tikzpicture}[rotate=-90,scale=.5]
	\tikzstyle{every node}=[minimum width=0pt, inner sep=2pt, circle]
\draw (180:1) node (v1) [draw] {};
		\draw (270:1) node (v3) [draw] {};
		\draw (360:1) node (v2) [draw] {};
		\draw (450:1) node (v4) [draw] {};
		\draw (v1) -- (v3);
		\draw (v1) -- (v4);
		\draw (v2) -- (v4);
	\end{tikzpicture}
&
	\begin{tikzpicture}[rotate=-90,scale=.7]
	\tikzstyle{every node}=[minimum width=0pt, inner sep=2pt, circle]
	\draw (-.5,-.9) node (v1) [draw] {};
	\draw (.5,-.9) node (v2) [draw] {};
	\draw (0,0) node (v3) [draw] {};
	\draw (-.5,.9) node (v4) [draw] {};
	\draw (.5,.9) node (v5) [draw] {};
	\draw (v1) -- (v2);
	\draw (v1) -- (v3);
	\draw (v2) -- (v3);
	\draw (v3) -- (v4);
	\draw (v3) -- (v5);
    \node[draw=none] at (0.9,0) {};
	\end{tikzpicture}
&
	\begin{tikzpicture}[scale=.7]
	\tikzstyle{every node}=[minimum width=0pt, inner sep=2pt, circle]
	\draw (-.5,0) node (v2) [draw] {};
	\draw (0,-.9) node (v1) [draw] {};
	\draw (.5,0) node (v3) [draw] {};
	\draw (1.5,0) node (v5) [draw] {};
	\draw (0,.9) node (v4) [draw] {};
	\draw (v1) -- (v2);
	\draw (v1) -- (v3);
	\draw (v2) -- (v3);
	\draw (v2) -- (v4);
	\draw (v3) -- (v4);
	\draw (v3) -- (v5);
	\end{tikzpicture}
\\
$P_4$
&
$\ltimes$
&
{\sf dart}\\
\end{tabular}
\end{center}
\caption{The minimal forbidden graphs for graphs with $\mz(G)\leq2$.}
\label{fig2}
\end{figure}

 The following result characterizes the graphs with $\mz(G) \leq 2$.

 \bigskip

\begin{thm}
\label{lemma:connectedmzandmrandgamma=2}\cite{BHL04,AIMZmr}
    Let $G$ be a graph.
    Then, the following are equivalent:
    \begin{enumerate}[label={\rm (\arabic*)}]
        \item $G^c$ has the form $(K_{s_1}\cup\cdots K_{s_t}\cup K_{p_1,q_1}\cup\cdots\cup K_{p_k,q_k})\vee K_r$ for non-negative integers $t,s_1,\dots,s_t,k,p_1,\dots,p_k,q_1,\dots,q_k,r$, with $p_i+q_i>0$ for $i=1,\dots,k$,
        \item $G$ is $\{P_4,\ltimes,{\sf dart},P_3\cup K_2,3K_2\}$-free,
        \item $\mz(G)\leq 2$.
    \end{enumerate}
\end{thm}

In \cite{BHL04}, it was observed that if $G$ is a connected graph, then $G$ is $\{P_4,\ltimes,{\sf dart}\}$-free if and only if $G$ is $\{P_4,P_3\cup K_1,K_2\cup 2K_1\}$-free.
From which follows the next specialization.

\begin{thm}
\label{lemma:mzandmrandgamma=2}\cite{BHL04,AIMZmr}
    Let $G$ be a connected graph.
    Then, the following are equivalent:
    \begin{enumerate}[label={\rm (\arabic*)}]
        \item $G^c$ can be expressed as the union of complete graphs and of complete bipartite graphs,
        \item $G$ is $\{P_4,\ltimes,{\sf dart}\}$-free,
        \item $\mz(G)\leq 2$.
    \end{enumerate}
\end{thm}

Let $\Forbmz{k}$ denote the set of minimal (under induced subgraphs) forbidden graphs for graphs with $\mz(G)\leq k$.
Previous characterizations rely on the forbidden graphs in $\Forbmz{k}$ to obtain the structure of the graphs with $\mz(G)\leq k$.
Our first result gives an upper bound on the number of vertices for every graph in $\Forbmz{k}$.

\medskip

%\textcolor{orange}{Escribir en la introducción o preeliminares la definición de forcing process, para que quede claro que los $a_i's$ del sigiuente teorema no empiezan necesariamente de color negro.}
\begin{thm}\label{thm_mz_fin}
Let $k\geq 1$ be an integer. Then, every graph in $\Forbmz{k}$ has at most $2k+2$ vertices.
\end{thm}
\begin{proof}
Suppose $G$ is a graph on $n$ vertices with $\mz(G)\geq k+1$.  Then $Z(G)\leq n-(k+1)$, which means that $G$ has a zero forcing set $B$ of cardinality $n-(k+1)$. Pick a forcing process $(a_i\rightarrow b_i)_{i=1}^{k+1}$ that starts with the set $B$, and let
\[
W:=\{a_i\}_{i=1}^{k+1}\cup\{b_i\}_{i=1}^{k+1}.
\]
Note that $|W|\leq 2k+2$.
Also, note that $B\cap W$ is a zero forcing set of the induced subgraph $G[W]$ of $G$, under the same forcing process. This means that $G[W]$ satisfies $\mz(G[W])\geq k+1$ and $V(G[W])=|W|\leq 2k+2$.

By repeating the same process, we conclude that any minimal forbidden graph for graphs with $\mz(G)\leq k$ has at most $2k+2$ vertices. 
\end{proof}

\begin{figure}[ht]
\begin{center}
\begin{tabular}{cccccccccccccccc}
\begin{tikzpicture}[scale=0.5,thick]
    \tikzstyle{every node}=[minimum width=0pt, inner sep=2pt, circle]
        \draw (0:1) node[draw] (0) {};
        \draw (72:1) node[draw] (2) {};
        \draw (144:1) node[draw] (3) {};
        \draw (216:1) node[draw] (1) {};
        \draw (288:1) node[draw] (4) {};
        \draw  (0) edge (2);
        \draw  (0) edge (4);
        \draw  (1) edge (3);
        \draw  (1) edge (4);
    %\draw (0,-1.5) node () {$P_5$};
    \end{tikzpicture}
    &
    \begin{tikzpicture}[scale=0.5,thick]
    \tikzstyle{every node}=[minimum width=0pt, inner sep=2pt, circle]
        \draw (0:1) node[draw] (0) {};
        \draw (60:1) node[draw] (1) {};
        \draw (120:1) node[draw] (2) {};
        \draw (180:1) node[draw] (3) {};
        \draw (240:1) node[draw] (5) {};
        \draw (300:1) node[draw] (4) {};
        \draw  (0) edge (4);
        \draw  (1) edge (4);
        \draw  (2) edge (5);
        \draw  (3) edge (5);
        \draw  (4) edge (5);
    %\draw (0,-1.5) node () {$H$};
    \end{tikzpicture}
    &
    \begin{tikzpicture}[scale=0.5,thick]
    \tikzstyle{every node}=[minimum width=0pt, inner sep=2pt, circle]
        \draw (0:1) node[draw] (0) {};
        \draw (60:1) node[draw] (4) {};
        \draw (120:1) node[draw] (1) {};
        \draw (180:1) node[draw] (3) {};
        \draw (240:1) node[draw] (2) {};
        \draw (300:1) node[draw] (5) {};
        \draw  (0) edge (4);
        \draw  (0) edge (5);
        \draw  (1) edge (4);
        \draw  (1) edge (5);
        \draw  (2) edge (5);
        \draw  (3) edge (5);
    %\draw (0,-1.5) node () {$R$};
    \end{tikzpicture}
    &
    \begin{tikzpicture}[scale=0.5,thick]
    \tikzstyle{every node}=[minimum width=0pt, inner sep=2pt, circle]
        \draw (0:1) node[draw] (1) {};
        \draw (60:1) node[draw] (0) {};
        \draw (120:1) node[draw] (2) {};
        \draw (180:1) node[draw] (3) {};
        \draw (240:1) node[draw] (5) {};
        \draw (300:1) node[draw] (4) {};
        \draw  (0) edge (4);
        \draw  (0) edge (5);
        \draw  (1) edge (4);
        \draw  (2) edge (5);
        \draw  (3) edge (5);
        \draw  (4) edge (5);
    %\draw (0,-1.5) node () {\texttt{E?qw}};
    \end{tikzpicture}
    &
    \begin{tikzpicture}[scale=0.5,thick]
    \tikzstyle{every node}=[minimum width=0pt, inner sep=2pt, circle]
        \draw (0:1) node[draw] (4) {};
        \draw (60:1) node[draw] (1) {};
        \draw (120:1) node[draw] (2) {};
        \draw (180:1) node[draw] (3) {};
        \draw (240:1) node[draw] (5) {};
        \draw (300:1) node[draw] (0) {};
        \draw  (0) edge (4);
        \draw  (0) edge (5);
        \draw  (1) edge (4);
        \draw  (1) edge (5);
        \draw  (2) edge (5);
        \draw  (3) edge (5);
        \draw  (4) edge (5);
    %\draw (0,-1.5) node () {\texttt{E?rw}};
    \end{tikzpicture}
    &
    \begin{tikzpicture}[scale=0.5,thick]
    \tikzstyle{every node}=[minimum width=0pt, inner sep=2pt, circle]
        \draw (0:1) node[draw] (3) {};
        \draw (60:1) node[draw] (1) {};
        \draw (120:1) node[draw] (4) {};
        \draw (180:1) node[draw] (2) {};
        \draw (240:1) node[draw] (0) {};
        \draw (300:1) node[draw] (5) {};
        \draw  (0) edge (4);
        \draw  (0) edge (5);
        \draw  (1) edge (4);
        \draw  (1) edge (5);
        \draw  (2) edge (4);
        \draw  (3) edge (5);
        \draw  (4) edge (5);
    %\draw (0,-1.5) node () {\texttt{E?zW}};
    \end{tikzpicture}
    &
    \begin{tikzpicture}[scale=0.5,thick]
    \tikzstyle{every node}=[minimum width=0pt, inner sep=2pt, circle]
        \draw (0:1) node[draw] (0) {};
        \draw (60:1) node[draw] (3) {};
        \draw (120:1) node[draw] (2) {};
        \draw (180:1) node[draw] (4) {};
        \draw (240:1) node[draw] (1) {};
        \draw (300:1) node[draw] (5) {};
        \draw  (0) edge (3);
        \draw  (0) edge (5);
        \draw  (1) edge (4);
        \draw  (1) edge (5);
        \draw  (2) edge (5);
        \draw  (3) edge (5);
    %\draw (0,-1.5) node () {\texttt{ECRo}};
    \end{tikzpicture}
    \\
    $\Gaa$ & $\Gab$ & $\Gac$ & $\Gad$ & $\Gae$ & $\Gaf$ & $\Gag$
    \\
    %Gaa & Gab & Gac & Gad & Gae & Gaf & Gag
    %\\
    \begin{tikzpicture}[scale=0.5,thick]
    \tikzstyle{every node}=[minimum width=0pt, inner sep=2pt, circle]
        \draw (0:1) node[draw] (1) {};
        \draw (60:1) node[draw] (5) {};
        \draw (120:1) node[draw] (2) {};
        \draw (180:1) node[draw] (3) {};
        \draw (240:1) node[draw] (0) {};
        \draw (300:1) node[draw] (4) {};
        \draw  (0) edge (3);
        \draw  (0) edge (4);
        \draw  (0) edge (5);
        \draw  (1) edge (4);
        \draw  (1) edge (5);
        \draw  (2) edge (5);
    %\draw (0,-1.5) node () {\texttt{ECr{\char`\_}}};
    \end{tikzpicture}
    &
    \begin{tikzpicture}[scale=0.5,thick]
    \tikzstyle{every node}=[minimum width=0pt, inner sep=2pt, circle]
        \draw (0:1) node[draw] (5) {};
        \draw (60:1) node[draw] (3) {};
        \draw (120:1) node[draw] (0) {};
        \draw (180:1) node[draw] (1) {};
        \draw (240:1) node[draw] (4) {};
        \draw (300:1) node[draw] (2) {};
        \draw  (0) edge (3);
        \draw  (0) edge (4);
        \draw  (0) edge (5);
        \draw  (1) edge (4);
        \draw  (2) edge (5);
        \draw  (4) edge (5);
    %\draw (0,-1.5) node () {Net};
    \end{tikzpicture}
    &
    \begin{tikzpicture}[scale=0.5,thick]
    \tikzstyle{every node}=[minimum width=0pt, inner sep=2pt, circle]
        \draw (0:1) node[draw] (4) {};
        \draw (60:1) node[draw] (0) {};
        \draw (120:1) node[draw] (3) {};
        \draw (180:1) node[draw] (2) {};
        \draw (240:1) node[draw] (5) {};
        \draw (300:1) node[draw] (1) {};
        \draw  (0) edge (3);
        \draw  (0) edge (4);
        \draw  (0) edge (5);
        \draw  (1) edge (4);
        \draw  (1) edge (5);
        \draw  (2) edge (5);
        \draw  (3) edge (5);
    %\draw (0,-1.5) node () {\texttt{ECro}};
    \end{tikzpicture}
    &
    \begin{tikzpicture}[scale=0.5,thick]
    \tikzstyle{every node}=[minimum width=0pt, inner sep=2pt, circle]
        \draw (0:1) node[draw] (1) {};
        \draw (60:1) node[draw] (5) {};
        \draw (120:1) node[draw] (2) {};
        \draw (180:1) node[draw] (3) {};
        \draw (240:1) node[draw] (0) {};
        \draw (300:1) node[draw] (4) {};
        \draw  (0) edge (3);
        \draw  (0) edge (4);
        \draw  (0) edge (5);
        \draw  (1) edge (4);
        \draw  (1) edge (5);
        \draw  (2) edge (5);
        \draw  (4) edge (5);
    %\draw (0,-1.5) node () {\texttt{ECrg}};
    \end{tikzpicture}
    &
    \begin{tikzpicture}[scale=.5,thick]
    \tikzstyle{every node}=[minimum width=0pt, inner sep=2pt, circle]
        \draw (0:1) node[draw] (0) {};
        \draw (72:1) node[draw] (3) {};
        \draw (144:1) node[draw] (2) {};
        \draw (216:1) node[draw] (1) {};
        \draw (288:1) node[draw] (4) {};
        \draw (0,0) node[draw] (5) {};
        \draw  (0) edge (3);
        \draw  (0) edge (4);
        \draw  (0) edge (5);
        \draw  (1) edge (4);
        \draw  (1) edge (5);
        \draw  (2) edge (5);
        \draw  (3) edge (5);
        \draw  (4) edge (5);
    %\draw (0,-1.5) node () {\texttt{ECrw}};
    \end{tikzpicture}
    &
    \begin{tikzpicture}[scale=0.5,thick]
    \tikzstyle{every node}=[minimum width=0pt, inner sep=2pt, circle]
        \draw (0:1) node[draw] (0) {};
        \draw (60:1) node[draw] (3) {};
        \draw (120:1) node[draw] (2) {};
        \draw (180:1) node[draw] (4) {};
        \draw (240:1) node[draw] (1) {};
        \draw (300:1) node[draw] (5) {};
        \draw  (0) edge (3);
        \draw  (0) edge (5);
        \draw  (1) edge (4);
        \draw  (1) edge (5);
        \draw  (2) edge (4);
        \draw  (2) edge (5);
        \draw  (4) edge (5);
    %\draw (0,-1.5) node () {\texttt{ECZg}};
    \end{tikzpicture}
    &
    \begin{tikzpicture}[scale=0.5,thick]
    \tikzstyle{every node}=[minimum width=0pt, inner sep=2pt, circle]
        \draw (0:1) node[draw] (5) {};
        \draw (60:1) node[draw] (3) {};
        \draw (120:1) node[draw] (0) {};
        \draw (180:1) node[draw] (2) {};
        \draw (240:1) node[draw] (4) {};
        \draw (300:1) node[draw] (1) {};
        \draw  (0) edge (3);
        \draw  (0) edge (4);
        \draw  (0) edge (5);
        \draw  (1) edge (4);
        \draw  (1) edge (5);
        \draw  (2) edge (4);
        \draw  (3) edge (5);
        \draw  (4) edge (5);
    %\draw (0,-1.5) node () {\texttt{ECzW}};
    \end{tikzpicture}
    \\
    $\Gah$ & $\Gai$ & $\Gaj$ & $\Gak$ & $\Gal$ & $\Gam$ & $\Gan$
    \\
    %Gah & Gai & Gaj & Gak & Gal & Gam & Gan
    %\\
    \begin{tikzpicture}[scale=0.5,thick]
    \tikzstyle{every node}=[minimum width=0pt, inner sep=2pt, circle]
        \draw (0:1) node[draw] (0) {};
        \draw (60:1) node[draw] (2) {};
        \draw (120:1) node[draw] (4) {};
        \draw (180:1) node[draw] (1) {};
        \draw (240:1) node[draw] (3) {};
        \draw (300:1) node[draw] (5) {};
        \draw  (0) edge (2);
        \draw  (0) edge (4);
        \draw  (0) edge (5);
        \draw  (1) edge (3);
        \draw  (1) edge (4);
        \draw  (1) edge (5);
        \draw  (2) edge (4);
        \draw  (3) edge (5);
    %\draw (0,-1.5) node () {\texttt{EQzO}};
    \end{tikzpicture}
    &
    \begin{tikzpicture}[scale=.5,thick]
    \tikzstyle{every node}=[minimum width=0pt, inner sep=2pt, circle]
        \draw (0:1) node[draw] (2) {};
        \draw (60:1) node[draw] (4) {};
        \draw (120:1) node[draw] (1) {};
        \draw (180:1) node[draw] (3) {};
        \draw (240:1) node[draw] (5) {};
        \draw (300:1) node[draw] (0) {};
        \draw  (0) edge (2);
        \draw  (0) edge (4);
        \draw  (0) edge (5);
        \draw  (1) edge (3);
        \draw  (1) edge (4);
        \draw  (1) edge (5);
        \draw  (2) edge (4);
        \draw  (3) edge (5);
        \draw  (4) edge (5);
    %\draw (0,-1.5) node () {\texttt{EQzW}};
\end{tikzpicture}
&
\begin{tikzpicture}[scale=0.5,thick]
    \tikzstyle{every node}=[minimum width=0pt, inner sep=2pt, circle]
        \draw (0:1) node[draw] (0) {};
        \draw (360/7:1) node[draw] (4) {};
        \draw (720/7:1) node[draw] (2) {};
        \draw (1080/7:1) node[draw] (3) {};
        \draw (1440/7:1) node[draw] (1) {};
        \draw (1800/7:1) node[draw] (5) {};
        \draw (2160/7:1) node[draw] (6) {};
        \draw  (0) edge (4);
        \draw  (0) edge (6);
        \draw  (1) edge (5);
        \draw  (1) edge (6);
        \draw  (2) edge (6);
        \draw  (3) edge (6);
        \draw  (4) edge (6);
        \draw  (5) edge (6);
    \end{tikzpicture}
   &
    \begin{tikzpicture}[scale=0.5,thick]
    \tikzstyle{every node}=[minimum width=0pt, inner sep=2pt, circle]
        \draw (270:0.4) node[draw] (2) {};
        \draw (0:1) node[draw] (6) {};
        \draw (1,1) node[draw] (3) {};
        \draw (90:1) node[draw] (0) {};
        \draw (90:0.4) node[draw] (4) {};
        \draw (180:1) node[draw] (5) {};
        \draw (270:1) node[draw] (1) {};
        \draw  (0) edge (4);
        \draw  (0) edge (5);
        \draw  (0) edge (6);
        \draw  (1) edge (5);
        \draw  (1) edge (6);
        \draw  (2) edge (5);
        \draw  (2) edge (6);
        \draw  (3) edge (6);
        \draw  (4) edge (5);
        \draw  (4) edge (6);
    \end{tikzpicture}
&
    \begin{tikzpicture}[scale=0.5,thick]
    \tikzstyle{every node}=[minimum width=0pt, inner sep=2pt, circle]
        \draw (1,1) node[draw] (3) {};
        \draw (270:1) node[draw] (1) {};
        \draw (270:0.4) node[draw] (2) {};
        \draw (180:1) node[draw] (5) {};
        \draw (90:1) node[draw] (4) {};
        \draw (90:0.4) node[draw] (0) {};
        \draw (0:1) node[draw] (6) {};
        \draw  (0) edge (4);
        \draw  (0) edge (5);
        \draw  (0) edge (6);
        \draw  (1) edge (5);
        \draw  (1) edge (6);
        \draw  (2) edge (5);
        \draw  (2) edge (6);
        \draw  (3) edge (6);
        \draw  (4) edge (5);
        \draw  (4) edge (6);
        \draw  (5) edge (6);
    \end{tikzpicture}
    &
    \begin{tikzpicture}[scale=0.5,thick]
    \tikzstyle{every node}=[minimum width=0pt, inner sep=2pt, circle]
        \draw (0:1) node[draw] (0) {};
        \draw (60:1) node[draw] (4) {};
        \draw (120:1) node[draw] (1) {};
        \draw (180:1) node[draw] (2) {};
        \draw (240:1) node[draw] (5) {};
        \draw (300:1) node[draw] (3) {};
        \draw  (0) edge (4);
        \draw  (1) edge (4);
        \draw  (2) edge (5);
        \draw  (3) edge (5);
    \end{tikzpicture}
    &
    \begin{tikzpicture}[scale=0.5,thick]
    \tikzstyle{every node}=[minimum width=0pt, inner sep=2pt, circle]
        \draw (0:1) node[draw] (0) {};
        \draw (60:1) node[draw] (3) {};
        \draw (120:1) node[draw] (1) {};
        \draw (180:1) node[draw] (4) {};
        \draw (240:1) node[draw] (2) {};
        \draw (300:1) node[draw] (5) {};
        \draw  (0) edge (3);
        \draw  (0) edge (5);
        \draw  (1) edge (4);
        \draw  (2) edge (5);
    \end{tikzpicture}
\\
 $\Gao$ & $\Gap$ & $\Gaq$ & $\Gar$ & $\Gas$ & $\Gat$ & $\Gau$
\\
%Gao & Gap & Gaq & Gar & Gas & Gat & Gau
%\\

    \begin{tikzpicture}[scale=0.5,thick]
    \tikzstyle{every node}=[minimum width=0pt, inner sep=2pt, circle]
        \draw (0:1) node[draw] (0) {};
        \draw (360/7:1) node[draw] (1) {};
        \draw (720/7:1) node[draw] (2) {};
        \draw (1080/7:1) node[draw] (6) {};
        \draw (1440/7:1) node[draw] (3) {};
        \draw (1800/7:1) node[draw] (5) {};
        \draw (2160/7:1) node[draw] (4) {};
        \draw  (0) edge (4);
        \draw  (1) edge (5);
        \draw  (2) edge (6);
        \draw  (3) edge (6);
    \end{tikzpicture}
    &
    \begin{tikzpicture}[scale=0.5,thick]
    \tikzstyle{every node}=[minimum width=0pt, inner sep=2pt, circle]
        \draw (0:1) node[draw] (5) {};
        \draw (360/7:1) node[draw] (1) {};
        \draw (720/7:1) node[draw] (2) {};
        \draw (1080/7:1) node[draw] (3) {};
        \draw (1440/7:1) node[draw] (4) {};
        \draw (1800/7:1) node[draw] (0) {};
        \draw (2160/7:1) node[draw] (6) {};
        \draw  (0) edge (4);
        \draw  (0) edge (6);
        \draw  (1) edge (5);
        \draw  (2) edge (6);
        \draw  (3) edge (6);
        \draw  (4) edge (6);
    \end{tikzpicture}
    &
    \begin{tikzpicture}[scale=0.5,thick]
    \tikzstyle{every node}=[minimum width=0pt, inner sep=2pt, circle]
        \draw (0:1) node[draw] (1) {};
        \draw (360/7:1) node[draw] (5) {};
        \draw (720/7:1) node[draw] (2) {};
        \draw (1080/7:1) node[draw] (3) {};
        \draw (1440/7:1) node[draw] (4) {};
        \draw (1800/7:1) node[draw] (0) {};
        \draw (2160/7:1) node[draw] (6) {};
        \draw  (0) edge (4);
        \draw  (1) edge (5);
        \draw  (1) edge (6);
        \draw  (2) edge (5);
        \draw  (2) edge (6);
        \draw  (3) edge (6);
        \draw  (5) edge (6);
    \end{tikzpicture}
    &
    \begin{tikzpicture}[scale=0.5,thick]
    \tikzstyle{every node}=[minimum width=0pt, inner sep=2pt, circle]
        \draw (0:1) node[draw] (0) {};
        \draw (45:1) node[draw] (4) {};
        \draw (90:1) node[draw] (2) {};
        \draw (135:1) node[draw] (3) {};
        \draw (180:1) node[draw] (1) {};
        \draw (225:1) node[draw] (5) {};
        \draw (270:1) node[draw] (7) {};
        \draw (315:1) node[draw] (6) {};
        \draw  (0) edge (4);
        \draw  (1) edge (5);
        \draw  (2) edge (6);
        \draw  (3) edge (7);
    \end{tikzpicture}
    \\
    $\Gav$ & $\Gaw$ & $\Gax$ & $\Gay$
    %\\
    %Gav & Gaw & Gax & Gay
\end{tabular}
\end{center}
\caption{Forbidden graphs for graphs with $mz\leq3$.}\label{fig:FG}
\end{figure}
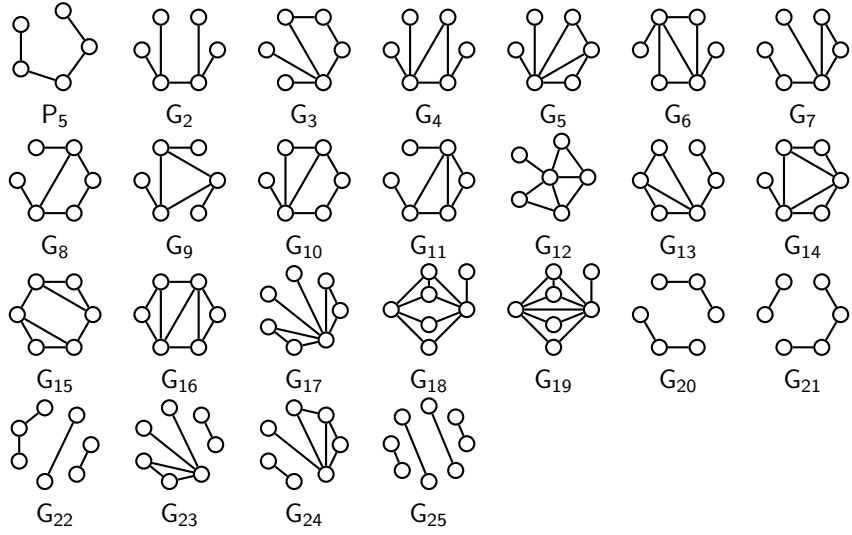
%{\color{violet} Aqui vamos a quitar los nombres Gaa, Gab, etc ?}

 As a corollary of Theorem \ref{thm_mz_fin}, it can be seen that the set $\Forbmz{k}$ is finite. 
 This is interesting since in the folklore\footnote{A result attributed to H. Tracy Hall is unpublished.} of the minimum rank problem, the number of minimal forbidden graphs with minimum rank at most 3 is infinite. 
 However, it is know that the family of graphs with minimum rank at most 3 are contained in the family of graphs with $\mz(G)$ at most 3. 
 An analogous example in the context of convex bodies is that we only need four linear constraints to describe a square, meanwhile, we need an infinite number of linear constraints to describe a circle contained in the square.

Due to the finiteness of $\Forbmz{k}$, with a brute-force search algorithm we were able to find $\Forbmz{3}$, that is, all the minimal forbidden graphs for graphs with $\mz(G)$ at most 3. The set of graphs in $\Forbmz{3}$ is depicted in Figure~\ref{fig:FG}.
Since the graphs with $\mz(G)$ at most 3 are the $\Forbmz{3}$-free graphs,  we have the following.

\begin{proposition}\label{mz3}
The graphs with $\mz(G)$ at most 3 do not contain any of the graphs shown in Figure~\ref{fig:FG} as induced subgraph. %{\color{red}add a few lines justifying why there are no more forbidden subgraphs}   
\end{proposition}

This leads us to the following problem.

\begin{problem}\label{prob}
Give an structural characterization of the graphs with $\mz(G)\leq 3$.
That is, find a description of the graphs having none of the graphs shown in Figure~\ref{fig:FG} as induced subgraph.
\end{problem}

In the rest of the paper, we focus on solving Problem \ref{prob}. 
Let us recall previous results regarding the graphs with $\mz(G)\leq 3$. 
Eroh, Kang and Yi \cite{eroh2015zero} characterized the trees and unicyclic graphs satisfying $\mz(G)=3$. 
More recently, Curl et al. \cite{ZF_complement} investigated the zero forcing number of the complement of different graph families, including trees, $K_{2,2}$-free bipartite graphs, unicyclic graphs, cactus graphs and $2$-trees. 
Their results show that the complement graph of many of these families satisfy $\mz(G)\leq 3$.

The following result firstly appeared in \cite{eroh2015zero}, but we include another proof using the minimal forbidden graphs.

\begin{proposition}
    Let $F$ be a forest with $\mz(F)$ at most three. 
    Then, $F$ is one, or a subgraph, of the following three forests:
    \begin{itemize}
        \item $K_{1,n}\cup K_2 \cup I$, 
        \item $3K_2\cup I$ or
        \item $T \cup I$,
    \end{itemize}
    where $I$ is an independent set and $T$ is the tree obtained from $K_{1,n}$ by subdividing one edge.
\end{proposition}
\begin{proof}
Let $F$ be a forest with $\mz(F)$ at most three.
By Proposition \ref{mz3}, $F$ does not contain any graph in Figure~\ref{fig:FG} as induced subgraph. 
Notice that, since $\Gaa$ is forbidden, then every nontrivial connected component of $F$ has diameter at most three. 
Furthermore, $F$ has at most three nontrivial connected components, since otherwise $\Gay$ would appear as an induced subgraph of $F$. 

Now, suppose that $F$ has three nontrivial connected components. 
Then $F\cong 3K_{2}\cup I$, where $I$ is an independent set of vertices, since otherwise $F$ contains $\Gav$ as an induced subgraph of $F$.

Suppose that $F=F_{1}\cup F_{2} \cup I$ has two nontrivial connected components and that $\operatorname{diam}(F_{1})=2$. 
Then, $F_{2}\cong K_{2}$, since otherwise $F_{2}$ contains an induced $P_{3}$, and consequently the forbidden graph $\Gat$ would appear as an induced subgraph of $F$, a contradiction. 
Therefore, $F\cong K_{1,n}\cup K_2 \cup I$.

Finally, assume that $F$ has a unique nontrivial connected component, this is $F=T\cup I$. 
If $\operatorname{diam}(F)=3$, then $T$ is the tree obtained from $K_{1,n}$ by subdividing one edge since otherwise $F$ contains the forbidden subgraph $\Gab$.
\end{proof}

The remaining sections are organized as follows. 
Recall that the {\it girth} of a graph is the length of a shortest cycle contained in the graph.
In Section~\ref{sec:mz<3g>=4}, we completely describe the graphs with $\mz(G)$ at most 3 and girth at least 4.
In Section~\ref{sec:mz<3g=3}, we study graphs with $\mz(G) \leq 3$ and girth 3, from which we derive structural descriptions. 
In Section~\ref{sec:allowable}, we present a computational approach to determine all allowable configurations and describe the resulting families.

\section{Graphs with $\mz$ at most 3 and girth at least 4}\label{sec:mz<3g>=4}

In this section, we completely describe the graphs with girth at least 4 and $\mz(G)$ at most 3.

% Notice that if $G$ has girth $g$, then $G$ has as an induced subgraph the cycle $C_g$.
% If we assume that $G$ has a girth greater than or equal to six, then $G$ would have as an induced subgraph a cycle of length at least $6$, which implies that $G$ contains the forbidden subgraph $P_5$.

\begin{proposition}
    There is no graph with girth at least $6$ and $\mz(G)\leq3$.
\end{proposition}

\begin{proof}
    Notice that if $G$ has girth $g$, then $G$ has as an induced subgraph the cycle $C_g$. Thus, if we suppose that $g\geq 6$, then $G$ would have as an induced subgraph a cycle of length at least $6$ which implies that $G$ contains the forbidden subgraph $P_5$.
\end{proof}

\begin{proposition}
If $G$ is a graph with girth $5$ and $\mz(G)\leq3$, then $G$ must be a $5$-cycle together with a (possible empty) set  of isolated vertices.
\end{proposition}

\begin{proof}
Let $C_5=\{v_1,v_2,v_3,v_4,v_5\}$ be an induced cycle of $G$, with edges $v_iv_{i+1}$ for $ 1\leq i \leq 4$ and $v_5,v_1$.
Suppose that some $v_i\in V(C_5)$ is adjacent to some vertex $u\in V(G)$ outside of $C_5$. 
Note that $u$ cannot be adjacent to any other vertex in $C_5$, since otherwise $G$ would contain a cycle of length smaller than $5$ which is impossible. 
Then, $G$ contains the $\Gaa$ induced subgraph with vertices $\{u,v_i,v_{i+1},v_{i+2},v_{i+3}\}$. 
Thus, each vertex of $C_5$ cannot be adjacent to any vertex outside $C_5$. 
If any other connected component than $C_5$ of $G$ contains an edge $xy$, then $G$ contains as an induced subgraph the forbidden graph $\Gau$ formed by the vertices $\{x,y,v_1, v_2,v_3,v_4\}$. 
Therefore, any other connected component different from $C_5$ must be an isolated vertex.
\end{proof}

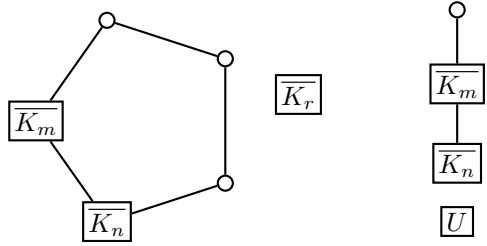
\begin{figure}[ht!]
\centering
		\begin{tikzpicture}[scale=0.7,thick]
		\tikzstyle{every node}=[minimum width=0pt, inner sep=2pt, circle]
			\draw (-1.96,0) node[draw,rectangle] (0) {$\overline{K_{m}}$ };
			\draw (-0.6180339887498951,-1.9021130325903073) node[draw,rectangle] (1) {$\overline{K_{n}}$ };
			\draw (1.618033988749894,-1.1755705045849465) node[draw] (2) { };
			\draw (1.6180339887498951,1.1755705045849458) node[draw] (3) { };
			\draw (-0.6180339887498946,1.9021130325903073) node[draw] (4) { };
			\draw (3,0.5) node[draw,rectangle] (5) { $\overline{K_{r}}$};
			\draw  (0) edge (1);
			\draw  (1) edge (2);
			\draw  (2) edge (3);
			\draw  (3) edge (4);
			\draw  (0) edge (4);

            \draw (6,-1.9) node[draw, rectangle] (9) { $U$};
			\draw (6,-0.8) node[draw, rectangle] (15) { $\overline{K_n}$ };
			\draw (6,0.7) node[draw, rectangle] (16) { $\overline{K_m}$};
			\draw (6,2.1) node[draw] (17) {};

			\draw  (15) edge (16);
			\draw  (16) edge (17);
            \end{tikzpicture}            
    \caption{The structure of a graph with girth $4$ and $\mz(G)$ at most three. The set $U$ is either an independent set or a $K_2$.} %\textcolor{orange}{La imagen parece más para el Lemma 8.}} \textcolor{purple}{Parece pero sí corresponde al teorema 9}
    \label{fig:C5K1}        
\end{figure}

\begin{theorem}
    If $G$ is a graph with girth $4$ and $\mz(G)\leq3$, then $G$ is graph isomorphic to any of the graphs shown in Figure~\ref{fig:C5K1}.
    %Assume $G$ has girth four. Then $G$ must be as in Figure~\ref{fig:C5K1}, where the vertices labeled $13, 24$ and $u$ represent independent sets and an edge means that there must exists every edge between the corresponding sets. %and a dashed edge means that there must exist either every edge or no edge between the corresponding sets. Moreover, one of the two dashed edges must correspond to the case where every edge exists.
\end{theorem}

\begin{proof}
Since $G$ has girth $4$, then $G$ contains an induced $C_4$, denoted by $C$. 
Let us assume that this 4-cycle consists of the vertex set $\{u_1,u_2,u_3,u_4\}$ with edges $u_1u_2$, $u_2u_3$, $u_3u_4$ and $u_1u_4$. 
For different elements $i,j\in\{1,2,3,4\}$, let $V_i$ 
be the set of vertices whose only neighbor in $C$ is the vertex $u_i\in V(C)$,
% adjacent to vertex $u_i\in V(C_4)$, as the only neighbor in the 4-cycle.
and let $V_{ij}$ be the set of vertices adjacent to $u_i$ and $u_j$. Also, let $U$ be the set of vertices that are not adjacent to any vertex in $C$.
Observe that the sets $V_{13}$ and $V_{24}$ are the only possible non-empty vertex sets adjacent to two vertices in $C$, because any other set of this type would give rise to a cycle of length $3$. Additionally, note that each set $V_{i}$ and each set $V_{ij}$ must be independent sets, as otherwise $G$ would have a cycle of length $3$. However, observe that each set $V_i$ must also induce a complete subgraph, since otherwise we obtain the forbidden subgraph $\Gac$. It follows that each set $V_i$ consists of a single vertex, from which this set will be simply referred to as a vertex.  Also note that there is no edge between $V_1$ or $V_3$ and any vertex in $V_{13}$, and there is also no edge between $V_2$ or $V_4$ and a vertex of $V_{24}$ since otherwise we obtain a $3$-cycle.

Observe that for any pair of different elements $i,j\in\{1,2,3,4\}$, there must exist an edge between $V_i$ and $V_j$, since otherwise we obtain the forbidden subgraph $\Gaa$ or $\Gah$. 
Therefore, there exist at most two of the vertices $V_i$, since otherwise the induced subgraph by the vertices $V_i$, $V_{j}$ and $V_k$ would induce a 3-cycle, which is impossible.
Moreover, if there exist two contiguous vertices $V_i$ and $V_{i+1}$ or $V_1$ and $V_4$, then $G$ would have the induced subgraph $\Gaa$ given by the vertices $V_i,V_{i+1},u_{i+1},u_{i+2},u_{i+3}$.
Therefore, if there exist two vertices $V_i, V_j$, these can only be $V_1$ and $V_3$, or $V_2$ and $V_4$.

% Now, let us assume that if there exists a vertex $V_i$, the first index is 1.
Now, assume that vertex $V_1$ exists. We claim that between $V_1$ and $V_{24}$ there must exist either every edge or no edge. To see this, suppose that $v_{24},v'_{24}$ are two different vertices in $V_{24}$, such that only one of them is adjacent to $V_1$, but then the induced subgraph given by the vertices $u_1,u_2,u_3,V_1,v_{24},v'_{24}$ would be isomorphic to the forbidden subgraph $\Gac$. In the same manner, if $V_3$ exists, then between $V_3$ and $V_{24}$ there must exist either every edge or no edge. However, if both $V_1$ and $V_3$ are non empty, then there exist every edge between $V_1$ and $V_{24}$, or there exist every edge between $V_3$ and $V_{24}$, because if $V_{24}$ has no edges between $V_1$ nor $V_3$, then the induced subgraph given by the vertices $V_3,V_1,u_1,u_2,v_{24}$ would be isomorphic to $\Gaa$. However, note that $V_1$ and $V_3$ cannot simultaneously have all adjacencies with $V_{24}$ since they are adjacent and would produce a 3-cycle. Analogously, if $V_2$ or $V_4$ exist, then between $V_2$ or $V_4$ and $V_{13}$ there must exist either every edge or no edge, and if both exist, then there exists every edge between $V_2$ and $V_{13}$, or there exists every edge between $V_4$ and $V_{13}$, but not all simultaneously. Now, we claim that if one set $V_i$ is non-empty, then there must exist every edge from $V_{24}$ to $V_{13}$ since otherwise we obtain the forbidden subgraph $\Gah$ or $P_5$ (see Figure \ref{girth42}). If $V_1$ is non-empty, then we obtain a $\Gah$ as the subgraph induced by the vertices $V_1, u_1,u_3,u_4,V_{13},V_{24}$ if $(V_1, V_{24})\notin E(G)$ or a $P_5$ induced by the vertices $V_1,u_2,u_3,V_{13},V_{24}$ if $(V_1, V_{24})\in E(G)$. Analogously for $i=2,3,4.$ %{\color{cyan} or $\Gaa$. To see this, first suppose that $V_1$ or $V_3$ are non empty and assume that there are two vertices $v_{24}\in V_{24}$ and $v_{13}\in V_{13}$ that are not adjacent. Let $i\in\{1,3\}$. If there are no edges between $V_i$ and $V_{24}$, then the induced subgraph given by the vertices $V_i,u_1,v_{13},u_3,u_4,v_{24}$ would be isomorphic to the forbidden subgraph $\Gah$. If there exists every edge between $V_i$ and $V_{24}$, then the induced subgraph given by the vertices $V_i,v_{24},u_2,u_3,v_{13}$ would be isomorphic to the forbidden subgraph $\Gaa$. A similar line of reasoning can be applied if $V_2$ or $V_4$ are non empty. This proves the claim.}

If we consider two vertices $x,y \in V_{13}$ and $v,w \in V_{24}$ then there are at least two edges between these vertices, since otherwise there exists a subgraph isomorphic to  either $\Gab$ or to $\Gah$ induced by the vertices $x,y,v,w,u_1,u_2.$ Moreover,  if there are exactly two, these two edges are not disjoint  since otherwise we obtain a $P_5$ induced by the vertices $x,y,v,w,u_1.$
Between $V_1$ and $V_{13}$ must exist every edge since otherwise we obtain the path $(V_1, u_1, V_{13}, u_3,u_2)$ which is  forbidden but this leads to a cycle of length three. Hence, either $V_{13}$ or $V_1$ are empty sets.

Now, observe that if $U$ contains an induced $P_3$, then the induced subgraph of this path together with the vertices $u_1,u_2,u_3$ is the forbidden subgraph $\Gat$. Also note that if $U$ contains two disjoint edges as induced subgraph, then the induced subgraph of these edges together with the vertices $u_1,u_2,u_3$ is the forbidden subgraph $\Gav$. Hence, the set $U$ is either an independent set, a complete subgraph, or a union of a complete subgraph with an independent set, but since the girth of $G$ is four, $U$ has at most one edge.
Observe that there are no edges between $U$ and any vertex $V_i$ for $i=1,2,3,4$, since otherwise we obtain the forbidden subgraph $\Gaa$. In fact, suppose that $u\in U$ is adjacent to some $V_i$, then, the induced subgraph given by the vertices $u,V_i,u_{i},u_{i+1},u_{i+2}$ is isomorphic to $\Gaa$.

Now, assume that one set $V_i$ is non-empty. In this case we have that any vertex of $U$ cannot be simultaneously adjacent to both a vertex of $V_{24}$ and a vertex of $V_{13}$, since every edge between $V_{13}$ and $V_{24}$ exists. Moreover, note that in this case, if some vertex of $u\in U$ is adjacent to a vertex of $V_{13}$ or $V_{24}$, then we obtain the forbidden subgraph $\Gaa$ as an induced subgraph. Finally, if $U$ contains an edge $xy$, we obtain the forbidden subgraph $\Gau$ as an induced subgraph by the vertices $u_1,v_{13},v_{24},u_4,x,y$, where $v_{13}\in V_{13}$ and $v_{24}\in V_{24}$. Thus, if one set $V_i$ is non-empty, then $U$ is an independent set.
\begin{figure}
    	\begin{center}
 		\begin{tikzpicture}[scale=1.7,thick]
 		\tikzstyle{every node}=[minimum width=0pt, inner sep=2pt, circle]
 			\draw (-2.45,0.74) node[draw] (0) {$u_1$};
 			\draw (-1.07,0.73) node[draw] (1) { $u_2$};
 			\draw (-2.44,-0.65) node[draw] (2) { $u_4$};
 			\draw (-1.03,-0.63) node[draw] (3) { $u_3$};
 			\draw (-0.55,1.87) node[draw] (4) { $V_1$};
 			\draw (0.41,0.3) node[draw] (5) { $V_3$};
 			\draw (-0.56,-1.32) node[draw] (6) { $V_{24}$};
 			\draw (-2.96,-1.34) node[draw] (7) { $V_{13}$};
 			\draw (0,-1.01) node[draw] (8) { $U$};
             %\draw (4) edge (7);
 			\draw  (0) edge (1);
 			\draw  (1) edge (3);
 			\draw  (2) edge (3);
 			\draw  (0) edge (2);
 			\draw  (0) edge (4);
 			\draw  (3) edge (5);
			\draw  (4) edge (5);
 			\draw  (1) edge (6);
 			\draw  (0) edge (7);
 			\draw  (3) edge (7);
 			\draw  (2) edge (6);
 			\draw  (4) edge (6);
             \draw (6) edge (7);

             \draw (4-2.45,0.74) node[draw] (0x) {$u_1$};
 			\draw (4-1.07,0.73) node[draw] (1x) { $u_2$};
 			\draw (4-2.44,-0.65) node[draw](2x) { $u_4$};
 			\draw (4-1.03,-0.63) node[draw](3x) { $u_3$};
 			\draw (4-0.55,1.87) node[draw] (4x) { $V_{1}$};
 			\draw (4.41,0.3) node[draw]   (5x) { $V_{3}$};
 			\draw (4-0.56,-1.32) node[draw](6x) { $V_{24}$};
 			\draw (4-2.96,-1.34) node[draw](7x) { $V_{13}$};
 			\draw (4,-1.01) node[draw] (8x) { $U$};
             %\draw (4) edge (7);
 			\draw  (0x) edge (1x);
 			\draw  (1x) edge (3x);
 			\draw  (2x) edge (3x);
 			\draw  (0x) edge (2x);
 			\draw  (0x) edge (4x);
 			\draw  (3x) edge (5x);
 			\draw  (4x) edge (5x);
 			\draw  (1x) edge (6x);
 			\draw  (0x) edge (7x);
 			\draw  (3x) edge (7x);
 			\draw  (2x) edge (6x);
 			\draw  (5x) edge (6x);
             \draw  (6x) edge (7x);
 		\end{tikzpicture}
 	\end{center}
     %los verices 6 7 y 8 son independientes
     \caption{The structure of a graph with girth four and $\mz(G)$ at most three.}
     \label{girth42}
 \end{figure}
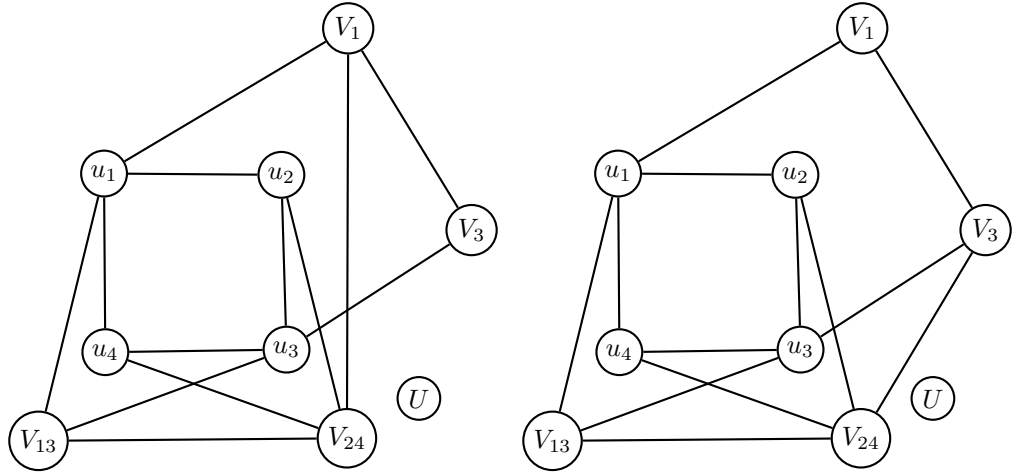
 
 Now assume that $V_1$ and $V_3$ are non-empty. Notice that  $N(u_2)=N(u_4)=N(V_{13})$ we may consider $u_2$ and $u_4$ to be a part of $V_{13}.$ If $V_{24}$ is adjacent to $V_1$, then  $N(u_1)=N(V_{24})$. Thus, we may consider $u_1$ to be part of $V_{24}$. Analogously, if $V_{24}$ is adjacent to $V_3$ then  $N(u_3)=N(V_{24})$, and thus we may consider $u_3$ to be part of $V_{24}$. In both cases, we obtain the left graph of Figure \ref{fig:C5K1}, and the same result is obtained if we assume that $V_2$ and $V_4$ are non-empty.
 If both $V_{13}$ and $V_{24}$ are non-empty and the set $U$ contains an edge, we obtain the forbidden subgraph $\Gau$.
 If each $V_i$ for $i=1,2,3,4$ is empty, and either $V_{13}$ or $V_{24}$ are empty, then the set $U$ is not necessarily an independent set since the rest of the graph does not contain an induced $P_4$ and therefore the graph does not contain the induced subgraph $\Gau$. 
 Since the rest of the graph contains $P_3$ as an induced subgraph, $U$ must be $P_3$-free, that is, $U$ is either an independent set or induces a complete subgraph. Hence, we obtain the right graph of Figure \ref{fig:C5K1}. It is easy to see that by allowing any other set to be empty, we do not obtain any new graph structures. 
\end{proof}

\section{Graphs with $\mz$ at most 3 and girth equal to 3}\label{sec:mz<3g=3}

In this section, we focus on the case in which $G$ has girth 3. We obtain partial characterizations of the graphs with $\mz(G)\leq 3$ and girth $3$. To do this, throughout the section, we assume that $G$ contains a $3$-cycle $C$ with vertex set $\{u_1,u_2,u_3\}$. For different elements $i,j\in\{1,2,3\}$, let $V_i$ 
be the set of vertices whose only neighbor in $C$ is the vertex $u_i$,
% adjacent to vertex $u_i\in V(C_4)$, as the only neighbor in the 4-cycle.
and let $V_{ij}$ be the set of vertices whose only neighbors in $C$ are $u_i$ and $u_j$. Also, let $V_t$ be the set of vertices adjacent to $u_1,u_2$ and $u_3$, and let $U$ be the set of vertices that are not adjacent to any vertex in $C$.

\begin{obs}\label{g3}
     Suppose that all $V_{i}$ are non-empty. If there are no edges between such sets, then the vertices $u_1,u_2$ and $u_3$ together with one neighbor from each $V_i$, induce the forbidden graph $\Gai$. Hence, there must exist edges between the sets $V_{i}$. However, if there is only one edge between these sets, say $(v_{1},v_{2})$, then the forbidden graph $\Gaa$ appears as the induced subgraph of vertices $v_{1},v_{2},u_2,u_3,v_{3}$. Therefore, if all $V_{i}$ are non-empty, then at least two sets $V_{i}$ and $V_{j}$ are adjacent. %{\color{blue} Añadí este párrafo que explica más detalladamente el de abajo. Por cierto que no me apareció una $G_{10}$ pero sí $P_{5}$ al considerar solo una arista entre los $V_{i}$.}    
    
    %Notice that at least one $V_i$ is an empty set since otherwise we obtain the forbidden subgraph $\hyperref[fig:FG]{G_9}$ or $\hyperref[fig:FG]{G_{10}}$ unless there exist at least two of edges $(V_1,V_2), (V_2,V_3)$ or $(V_3,V_1).$
\end{obs}

\begin{remark}\label{rem1}
    In view of Observation \ref{g3}, we are going to assume in what remains of this section that each set $V_i$ is non-empty, and that every edge between $V_2$ and $V_i$ exists for $ i=1,3$. Moreover, unless otherwise specified, we will also assume that there are no edges between $V_1$ and $V_3$.
    %{\color{cyan} Por qué se asume que no existe arista entre $V_1$ y $V_3$?} {\color{violet} Pues sabemos que entre $V_1, V_2$ y $V_3$ hay al menos dos de las tres aristas. Así que decidí empezar con el caso en que hay exactamente dos aristas. Por lo tanto asumo sin pérdida de generalidad que no hay aristas entre $V_1$ y $V_3.$}
\end{remark}

\begin{lemma}
If $G$ has girth 3 and at least two sets $V_i$ and $V_j$ are non-empty,  then $U$ is an independent set.
\end{lemma}

\begin{proof}
    Assume first that $V_1 \neq \emptyset \neq V_3$ and that there are no edges between $V_1$ and $V_3.$
    Let $v_1\in V_1$ and $v_3\in V_3$. Since we assume that $(v_1,v_3)\notin E(G)$, note that $G$ contains an induced $P_4$ given by the vertices $v_1, u_1, u_3, v_3$. Thus, if $U$ is adjacent to either $V_1$ or $V_3$ but not both, we obtain the forbidden subgraph $\Gaa.$ However, if there is some $u\in U$ adjacent to both $v_1\in V_1$ and $v_3\in V_3$, then we obtain the forbidden subgraph $\Gaa$ as an induced subgraph given by the vertices $v_3, u, v_1,u_1,u_2.$ This implies that $U$ is not adjacent to $V_1$ nor to $V_3$. Thus, if $U$ contained an edge, then this edge together with the induced $P_4$ gives the forbidden subgraph $\Gau$ as an induced subgraph in $G$. Hence, $U$ is an independent set.

    Now assume that $V_1 \neq \emptyset \neq V_2$ and that every edge exists between $V_1$ and $V_2$. Let $v_i \in V_i$ for $i=1,2$ and $u\in U$. If $(u,v_1), (u,v_2)$ then $G$ contains the forbidden subgraph $\Gao$ induced by the vertices $u, v_1, v_3, u_1, u_2, u_3.$ If $(u, v_i) \in E(G)$ for some $i=1,2$ but $(u, v_j)\notin E(G)$ for $i \neq j  =1,2$ then $G$ contains a $P_5$ induced by the vertices $u, v_i, v_j, u_j, u_3$. Thus $(u,v_1), (u, v_2)\notin E(G).$ If $U$ contains an edge, then together with the $P_4$ induced by the vertices $u_3, u_2,v_2,v_1$, we obtain the forbidden subgraph $\Gau$.   
\end{proof}

\begin{obs}
%{\color{cyan} En el enunciado del lema se debería agregar que se asume que se tienen al menos $u_2$ conjuntos $V_i$ no vacíos. Además de esto, falta considerar que si se asume que al menos un $V_i$ es no vacío y no es adyacente a $U$, entonces $U$ tendria que ser independiente o una completa, porque de otra forma se obtendría la $\Gat$. Y creo que este es el único caso en el que sale que $U$ es una completa, pues en los casos considerados en la prueba como está $U$ sólo es un conjunto independiente y no una completa como se enuncia.} 
%{\color{violet}Ya correji el enunciado del lema y agregue el caso en que los dos $V_i$ no vacíos son adyacentes entre sí.  Que te parece dejar el caso cuando solo uno es no vacio y no adyacente a U  como observacion:}
If at least one set $V_i$ is non-empty, and not adjacent to $U$, then $U$ may also be a complete subgraph.
\end{obs}

\begin{theorem}\label{aristas}
    Assume that $G$ is a graph of girth three,  let  $V_i$ and $V_{ij}$ be the sets as in Remark \ref{rem1}. Then $G$ has the following form, where, every lilac edge may or may not exist,  every pink vertex is a complete subgraph, and the orange vertex is either a complete or an independent set.
    Moreover there are at least three lilac edges; and if there are exactly three lilac edges, these are $(V_{12},V_{13}),(V_{12},V_{23})$ and $(V_{13},V_{23})$.
    %to be specific, every case with seven or eight lilac edges is possible, the only case possible with three lilac edges is $\{(12,13), (13,23), (12,23)\}$, there are four possible cases with four lilac edges, twelve possible cases with five lilac edges, and only three forbidden cases with six lilac edges, all of which shall be revealed in the proof of the theorem.
   % {\color{magenta} Faltan las formas} {\color{violet} son las 43 que les mencion\'e, o podemos reescribir el enunciado del teorema. Podemos decidir despues de ver cuantos casos particulares tenemos dibujados. Por lo pronto reescribi el teorema.}

    
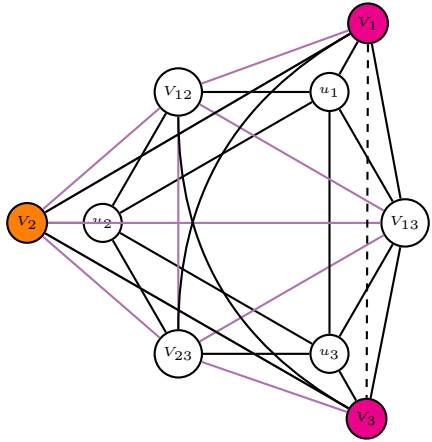
\begin{center}
		\begin{tikzpicture}[scale=1,thick]
		\tikzstyle{every node}=[minimum width=0pt, inner sep=2pt, circle]
			\draw (-2,0) node[draw] (0) { \tiny $u_2$};
			\draw (-1,-1.732050807568877) node[draw] (1) { \tiny $V_{23}$};
			\draw (1,-1.7320508075688776) node[draw] (2) { \tiny $u_3$};
			\draw (2,0) node[draw] (3) { \tiny $V_{13}$};
			\draw (1.000000000000001,1.732050807568877) node[draw] (4) { \tiny $u_1$};
			\draw (-1,1.7320508075688772) node[draw] (5) { \tiny $V_{12}$};
			\draw (1.51,2.64) node[draw, fill=magenta] (6) { \tiny $V_1$};
			\draw (-3,0) node[draw, fill=orange] (7) { \tiny $V_2$};
			\draw (1.49,-2.59) node[draw, fill=magenta] (8) { \tiny $V_3$};
			\draw  (0) edge (1);
			\draw  (0) edge (2);
			\draw  (0) edge (4);
			\draw  (0) edge (5);
			\draw  (1) edge (2);
			\draw[Orchid]  (1) edge (3);
			\draw[Orchid]  (1) edge (5);
			\draw  (2) edge (3);
			\draw  (2) edge (4);
			\draw  (3) edge (4);
			\draw[Orchid]  (3) edge (5);
			\draw  (4) edge (5);
			\draw  (4) edge (6);
			\draw  (0) edge (7);
			\draw  (2) edge (8);
			\draw  (3) edge (6);
			\draw  (3) edge (8);
            
			\draw[Orchid]  (8) edge (7);
			\draw[Orchid]  (6) edge (7);
            
			\draw[Orchid]  (5) edge (7);
			\draw[Orchid]  (1) edge (7);
            \draw[Orchid]  (3) edge (7);

            \draw[Orchid]  (5) edge (6);
			\draw[Orchid]  (1) edge (8);
            
			%\draw[dashed]  (1) edge (8);
			%\draw[dashed]  (5) edge (6);
			\draw  (6) edge (7);
			\draw  (7) edge (8);
			\draw[dashed]  (6) edge (8);
			\draw[bend left]  (1) edge (6);
			\draw[bend right]  (5) edge (8);
			%\draw[dotted, bend left]  (3) edge (7);
		\end{tikzpicture}
	\end{center}
\end{theorem}
\begin{proof}
%From a vertex in $V_{ij}$, there must exist either every edge or no edge to the vertices of $V_i$ since otherwise we obtain either a $P_5$ or the forbidden subgraph $\hyperref[fig:FG]{G_{12}}.$ {\color{violet} esto no lo especifica el teorema pero podriamos agregarlo. el problema es que para las demas aristas no se si son todas o ninguna (no se si con el programa se sepa).}

Notice first that both $V_1$ and $V_3$ must be complete subgraphs since otherwise we obtain the forbidden subgraph $\Gad$. We also claim that $V_2$ is either a complete graph or an independent set, since an induced $P_3$ or an induced $K_2\cup K_1$ in $ V_2$ lead to the forbidden subgraphs $\Gam$ or $\Gag$ respectively,  induced by the three vertices in $V_2$ together with the vertices $u_1,u_3,$ and $v_3\in V_3$. % In fact, suppose that $v_2,v_2',w$ are tree vertices in $V_2$ that do not induce a $3$-cycle nor are independent. Thus, assume that $(v_2,v_2')\in E(G)$, and let $v_1\in V_1$. If there is one more edge between such vertices, say $(v_2,w)$, then we obtain the forbidden subgraph $\Gam$ as an induced subgraph given by the vertices $v_2,v_2',w,v_1,u_1,u_3$; if $(v_2,w),(v_2)\notin E(G)$, then we obtain the forbidden subgraph $\Gag$ as an induced subgraph given by the vertices $v_2,v_2',w,v_1,u_1,u_3$. Hence $V_2$ is a complete graph or an independent set as claimed.

Now we claim that there must exist every edge between $V_1$ and $V_{13}$, and every edge between $V_3$ and $V_{13}$, since otherwise we obtain the forbidden subgraph $\Gaf$ or $\Gan$ as induced subgraph. To see this, let $v_{13}\in V_{13}$, $v_1\in V_1$ and $v_3\in V_3$. Note that if there is no edge between such vertices, then we obtain the subgraph $\Gaf$ as induced by the vertices $u_1,u_2,u_3,v_{13},v_1,v_3$; if only $v_1$ or $v_3$ is adjacent to $v_{13}$, say $v_1$, then  we obtain the subgraph $\Gan$ as induced by the vertices $u_1,u_2,u_3,v_{13},v_1,v_3$. 

 We claim that there must exist every edge between $V_{23}$ and $V_1$ since otherwise there exists an induced $\Gaa$ or $\Gak$ Let $v_{23}\in V_{23}$ and $v_i\in V_i$ for $i=1,2,3$.  If $(v_3, v_{23})\in E(G)$ then the $\Gaa$ is induced by the vertices $(v_3, v_{23}, u_2,u_1,v_1)$. If $(v_3,v_{23})\notin E(G)$ then the $\Gak$ is induced by the vertices $v_1,v_3,u_1,u_2,u_3,v_{23}.$

Analogously, there must exist every edge between $V_3$ and $V_{12}.$  If either the edge $(V_{2},V_{12})$ or the edge $(V_{2}, V_{23})$ does not exist, then $V_2$ must be a complete subgraph since otherwise we obtain the forbidden subgraph $\Gae$.

    %solo en los casos en que existen ambas v_2,12 y V_2,23 es posible que v2 no sea completa

Between the sets $\{V_1, V_2, V_3, V_{12}, V_{13}, V_{23}\}$, when considered to be of cardinality one, there are eight possible edges which might or might not exist. %We are going to show which are all possible combinations. 

Notice that if there exists a unique lilac edge, we obtain the forbidden subgraphs as we can see in Table \ref{1}. The cases where the unique lilac edge is $(V_{23}, V_2),(V_{23}, V_3),(V_{23}, V_{13})$ are analogous.

\begin{table}[h!]
    \centering
    \begin{tabular}{|c|c|} \hline
    lilac edges& forbidden subgraph \\ \hline
     $(V_1,V_{12}) $  & $\Gac$ with vertices $\{V_{13},u_3,V_1,V_{23},V_2,V_{12}\}$  \\
      $(V_2,V_{12})$  & $\Gag$ with vertices $\{u_3, u_1,V_{23},V_2,V_{12},V_{13}\}$ \\
      $(V_{12},V_{13})$  &  $\Gak$ with vertices $\{V_2,V_{13},u_3,V_{12},V_3, V_{23}\}$ \\
      $(V_{12}, V_{23})$ & $ \Gah$ with vertices $\{V_{12}, V_{23}, u_1, V_{13}, V_3,V_2\}$ \\
      $(V_2, V_{13})$ & $\Gai$ with vertices $\{V_2, V_{13},V_{12}, V_{23}, u_1, u_3$\}\\\hline 
      %$(12,13), (V_3,23)$ & $G_8$ with vertices $\{1,3,V_1,23,V_2,12\}$
    \end{tabular}
    \caption{The four cases in which there  exists only one of the eight lilac edges and the forbidden subgraphs we obtain.}
    \label{1}
\end{table}

Notice first that there must exist at  least two lilac edges, and one of them is either $(V_1, V_{12})$ or $(V_{12},V_{13})$ (analogously, we must have either $(V_3, V_{23})$ or $(V_{13}, V_{23})$). This means that if there are exactly two lilac edges, we only have three cases, all of which lead to forbidden subgraphs as we can see in Table \ref{2aristas}.
\begin{table}[h!]
    \centering
    \begin{tabular}{|c|c|} \hline
    lilac edges & forbidden subgraph \\ \hline
     $(V_1,V_{12}), (V_{13}, V_{23})$    & $\Gah$ with vertices $\{u_1,u_3,V_3,V_{23},V_2,V_{12}\}$  \\
      $(V_1,V_{12}), (V_3,V_{23})$   & $\Gac$ with vertices $\{u_3, V_1,V_{23},V_2,V_{12},V_{13}\}$ \\
      $(V_{12},V_{13}), (V_{13},V_{23})$ & $\Gaa=(V_1,V_{23},u_2,V_{12},V_3)$ \\ \hline 
      %$(12,13), (V_3,23)$ & $G_8$ with vertices $\{1,3,V_1,23,V_2,12\}$
    \end{tabular}
    \caption{The three cases in which there exist only two of the eight lilac edges and the forbidden subgraphs we obtain.}
    \label{2aristas}
\end{table}

This means that there exist at least three of the eight lilac edges. Notice that the three edges cannot be one of the following sets: $$\{(V_1, V_{12}),(V_{13},V_{23}),(V_2,V_{23}) \}, \{(V_1,V_{12}),(V_3,V_{23}), (V_{12},V_{13})\}, \{(V_{12},V_{13}),(V_{13},V_{23}), (V_3,V_{23})\}$$ nor $\{(V_{12},V_{13}),(V_{13},V_{23}),(V_1,V_{12})\}.$ Thus the remaining sets of three edges that we need to consider are shown in Table \ref{3aristas}.
\begin{table}[h!]
    \centering
    \begin{tabular}{|c|c|} \hline
    Three lilac edges & forbidden subgraph \\ \hline
       $(V_{12},V_{13}),(V_{13},V_{23}),(V_{12},V_{23})$  & allowable \\
        $(V_1,V_{12}), (V_3,V_{23}), (V_{12},V_2)$ & $\Gaa=(V_{23},u_3,u_1,V_{12},V_2)$ \\
        $(V_1,V_{12}),(V_{13},V_{23}),(12,23)$ & $\Gan$ with vertices $\{u_1,u_3,V_{12},V_1,V_2,V_{13}\}$ \\
        $(V_1,V_{12}),(V_{13},V_{23}),(V_{12},V_2)$ & $\Gaa=(V_{23},u_3,u_1,V_{12},V_2)$ \\ \hline
        
    \end{tabular}
    \caption{The four cases in which there are exactly three of the eight edges}
    \label{3aristas}
\end{table}

    % \href{https://sagecell.sagemath.org/?z=eJyFlf1vmkAYx3838X-4tT8AKTE7WN_MiO20Yrq02ZKNvThmEQ4kVc6AtllM97fvOQ7lTs4OE4H7fu55ByISo5AulusVGd7prtFttxAcx2hI82kaRSRDSR4sZwWKab5Adz_RewfZHFqk2SQusVNn7DJK1wYfEs1EWbAgyEHap1PN8Pfhsx1807t3Bdq1NMPkNDu20NXtNxGylVDvXoLeHYCeRehUDfkSdKaO6UGCztXQDymmCzX0VYIulZArZ4ffKqnrOykojNUOPZlSF73_XabUVb_6LcelLvstSSTqQN0plWZBnWP_80Si1Dl-lHuIzxRjeL4bw2GvH9c84Mp2DnuDv1SklP0c9h6u5AiV9QVsIKerLDBgU6l2lrLCLDQpgUtFvhe7fF2IsScHWT-n_D-eBwkoX_I1abeqJZqjEUoz8bFHJ-JzLd2dS3cX3TroARh2O8V6Wr5YJgUJ8nCmj0ywHa1DEjnMq1HzaQxb3jjonmZEMMOO42WeZit91GGZ6Aa4PILfCSpWuT7oPJF8lYak0Auar7hVw5ANyPv3xEGnmNFnfW-1Ks0wmBdEVqY5CR75EsTMOCFc7unoej6nzyTiL9UjgxW33fKw6VmmZzuahzVT8yz2Z2vtVuLwlm083B0Dgm3Tsk3sm9jqjuEW9viwFTQQfNC64_LCgwtgQRQixGyRreGKgn1wB6v-CwQyT-dBCImNdQgHW4apM4fV2bLhDK49u1pnZ_BQ6rjSy-i2c8RnKWKflxkJH_UwKEix_bywWWILbJxKoVEorQ_LXQRDynShBckkpMs_EGfSYRd6Q-oEUTQhUQJ9H5PSU-mGpwd9ISijq61jX9hefwa5IaMR069MKzsWp0-Eu2D12hwoxIu5OVBKlcKL-yI0rMm8ahdXyl7Ttopkd4_Bzd22vLup4FfsYrthF_9P2Q0VV_yyznSdi3XetyLssV7zUxPC47CpxYPmVETdaBYkH-46UOMfnwE0nQ==&lang=sage&interacts=eJyLjgUAARUAuQ==}{A code to compute allowable graphs} {\color{violet} este codigo va en la siguiente seccion}
  %{\color{purple} Hasta ahora hay 51 casos posibles, me gustaria que quedara el dibujo de todos ellos pero tal ves sean demaciados (igual y se reducen por simetría)}
\end{proof}

\begin{remark}\label{VtVi}

    The set $V_t$ must be adjacent to at least one of the sets $V_1, V_2$ or $V_3$ since otherwise we obtain the path $(V_t,3,V_3,V_2,V_1).$
 \end{remark}

 \begin{remark}
         If there are no edges between $V_{13}$ and $V_2$ then $V_{13}$ must be either a complete subgraph or an independent set since otherwise we obtain either $G_{13}$ or $G_4.$
     If there are no edges between $V_{12}$ and $V_1\cup V_2$ then $V_{12}$ must be a independent set since otherwise we obtain a $G_4$. Analogously, if there are no edges between $V_{23}$ and $V_2\cup V_3,$ then $V_{23}$ must be an independent set. If there are no edges between $V_{12}$ and $V_2$ but there exists every edge between $V_1$ and $V_{12}$, then $V_{12}$ must be a complete subgraph. Analogously for $V_{23}.$
 \end{remark}

 \begin{lemma} \label{3a}
     Assume there are only edges between the sets $V_{12}, V_{13}$ and $V_{23}$.  Then $G$  is either as in Figure \ref{3} where each lilac vertex is an independent set, each dark pink (magenta) is a complete subgraph and each orange vertex is either a complete or an independent set, or, the graph is as in the right side of Figure \ref{3}, with the following differences:
     Let $n(V_t)$ denote the number of sets ($V_i, V_{ij},\dots,$ etc) that are adjacent to $V_t$
     \begin{itemize}
         \item  $n(V_t)\geq 5$ with $V_{12}, V_{23}\subset N(V_t),$
         \item if $n(V_t)=9$ then $V_t$ may take any form (without forbidden subgraphs) 
         \item if $n(V_t)<9$ then $V_t$ does not contain as induced subgraphs the graphs $P_4, \ltimes, P_3\cup P_2, dart, 3K_2$ if $n(V_t)=8.$   
     \end{itemize}

     %Then $V_t$ must be adjacent to the sets $(12), (23), V_1$ and $V_3$, the set $V_2$ must be a complete subgraph and the sets $(12)$ and $(23)$ must be independent sets. The set $(13)$ is either a complete subgraph or an independent set. and the set $V_t$ does not contain as an induced subgraph $K_2\cup 2K_1$. and it can additionally be adjacent to $V_2$ and/or $(13).$ 

 \end{lemma}

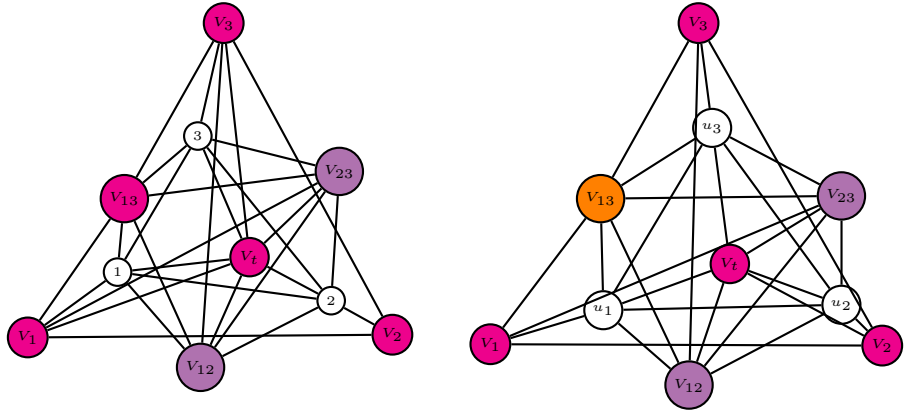
\begin{figure}
 	\begin{center}
 		\begin{tikzpicture}[scale=1.8,thick]
 		\tikzstyle{every node}=[minimum width=0pt, inner sep=2pt, circle]
 			\draw (0.15,0.47) node[draw, fill=magenta] (0) { \tiny $V_t$};
 			\draw (-0.82,0.36) node[draw] (1) { \tiny 1};
 			\draw (0.75,0.15) node[draw] (2) { \tiny 2};
 			\draw (-0.23,1.36) node[draw] (3) { \tiny 3};
 			\draw (-1.48,-0.12) node[draw, fill=magenta] (4) { \tiny $V_1$};
 			\draw (1.2,-0.1) node[draw, fill=magenta] (5) { \tiny $V_2$};
 			\draw (-0.04,2.19) node[draw, fill=magenta] (6) { \tiny $V_3$};
 			\draw (-0.77,0.9) node[draw, fill=magenta] (7) { \tiny $V_{13}$};
 			\draw (0.81,1.1) node[draw, fill=Orchid] (8) { \tiny $V_{23}$};
 			\draw (-0.21,-0.34) node[draw, fill=Orchid] (9) { \tiny $V_{12}$};
 			\draw  (0) edge (1);
 			\draw  (0) edge (2);
 			\draw  (0) edge (3);
 			\draw  (1) edge (2);
 			\draw  (1) edge (3);
 			\draw  (2) edge (3);
 			\draw  (3) edge (6);
 			\draw  (1) edge (4);
 			\draw  (2) edge (5);
 			\draw  (5) edge (6);
 			\draw  (4) edge (5);
 			\draw  (3) edge (7);
 			\draw  (1) edge (7);
 			\draw  (4) edge (7);
 			\draw  (6) edge (7);
 			\draw  (3) edge (8);
 			\draw  (2) edge (8);
 			\draw  (1) edge (9);
 			\draw  (2) edge (9);
 			\draw  (6) edge (9);
 			\draw  (4) edge (8);
 			\draw  (7) edge (8);
 			\draw  (8) edge (9);
 			\draw  (7) edge (9);
 			\draw  (0) edge (9);
 			\draw  (0) edge (8);
 			\draw  (0) edge (6);
 			\draw  (0) edge (4);

             \draw (3.68,0.42) node[draw, fill=magenta]  (0x) { \tiny $V_t$};
 			\draw (3.5-0.75,0.08) node[draw] (1x) { \tiny $u_1$};
 			\draw (4.5,0.12) node[draw]  (2x) { \tiny $u_2$};
 			\draw (3.55,1.42) node[draw]  (3x) { \tiny $u_3$};
 			\draw (3.5-1.58,-0.17) node[draw, fill=magenta](4x) { \tiny $V_1$};
 			\draw (4.8,-0.18) node[draw, fill=magenta]  (5x) { \tiny $V_2$};
 			\draw (3.45,2.19) node[draw, fill=magenta] (6x) { \tiny $V_3$};
 			\draw (3.5-0.77,0.9) node[draw, fill=orange]  (7x) { \tiny $V_{13}$};
 			\draw (4.5,0.92) node[draw, fill=Orchid]  (8x) { \tiny $V_{23}$};
 			\draw (3.5-0.12,-0.47) node[draw, fill=Orchid](9x) { \tiny $V_{12}$};
 			\draw  (0x) edge (1x);
 			\draw  (0x) edge (2x);
 			\draw  (0x) edge (3x);
 			\draw  (1x) edge (2x);
 			\draw  (1x) edge (3x);
 			\draw  (2x) edge (3x);
 			\draw  (3x) edge (6x);
 			\draw  (1x) edge (4x);
 			\draw  (2x) edge (5x);
 			\draw  (5x) edge (6x);
 			\draw  (4x) edge (5x);
 			\draw  (3x) edge (7x);
 			\draw  (1x) edge (7x);
 			\draw  (4x) edge (7x);
 			\draw  (6x) edge (7x);
 			\draw  (3x) edge (8x);
 			\draw  (2x) edge (8x);
 			\draw  (1x) edge (9x);
            \draw  (2x) edge (9x);
 			\draw  (6x) edge (9x);
 			\draw  (4x) edge (8x);
 			\draw  (7x) edge (8x);
 			\draw  (8x) edge (9x);
 			\draw  (7x) edge (9x);
 			\draw  (0x) edge (5x);
 			\draw  (0x) edge (9x);
 			\draw  (0x) edge (8x);
 		\end{tikzpicture}
 	\end{center} \caption{The graph $G$ when there are only edges between the sets $V_{12},V_{13}$ and $V_{23}.$} \label{3}
    \end{figure}

 \begin{proof}
 
 From Table \ref{caso3a}, we can see that the options for $N(V_t)$ with $|N(V_t)|<8$ are either, $\{V_2, V_{12}, V_{23}\}\cup \{1,2,3\}$, or a set containing $V_1$ and $V_3$ in which case, it must also contain $V_{12}$ and by symmetry, $V_{23}.$ 
 Now assume that $d(V_t)=8,$ if either $V_{12}$ or $V_{23}$ is not adjacent to $V_t$, we obtain the forbidden subgraph $G_{14}$ formed by the vertices $V_2,V_1,V_t,1,3,V_{12}$ or $V_2,V_3,V_t,2,1,V_{23}$ respectively.
 \begin{table}[h!]
         \centering
         \begin{tabular}{|c|c|} \hline
             $N(V_t)\setminus\{u_1,u_2,u_3\}$ & Forbidden subgraph  \\ \hline
            $V_2$  & $\Gaa=(V_t,u_2,V_{12},V_{13},V_i), i=1,3$ \\
            $V_3$ or $V_2,V_3$ & $\Gaa=(V_t, u_2, V_{23}, V_{13}, V_1)$ \\
           $V_2,V_{13}$ & $\Gaa=(V_t,V_2,V_3,V_{12},V_{23}$ \\
            $V_1,V_2,V_3$ & $\Gaa=(V_2,V_t,u_1,V_{1i},V_{23}), i=2,3$ \\
            $V_1,V_3$ & $\Gah[V_t,u_1,V_2,V_3,V_{12},V_{23}]$ \\
            $V_2,V_{12},V_{23},V_{13}$ & $\Gan[V_t,u_2,u_3,V_{13},V_1,V_3]$ \\
            $V_1,V_2,V_{12},V_{23}$ & $\Gaa=(u_2,V_t,V_1,V_{13},V_3)$ \\ \hline
         \end{tabular}
         \caption{Some non-possible options for $N(V_t).$}
         \label{caso3a}
     \end{table}
     If there are no edges between $V_1$  and $V_t$ we obtain the path $(V_t, u_2, V_{23}, V_{13}, V_1)$ which means that $V_t$ is adjacent to both $V_1$ and $V_3$, or is additionally adjacent to $V_{23}$ and or $V_{13}.$ If $V_t$ is adjacent to $V_1$, $V_2$ and $V_3$ we obtain the path $(V_2,V_t,u_1,V_{13},V_{23})$ and if $V_t$ is only adjacent to $V_1$ and $V_3$ we obtain the forbidden subgraph $G_4$ formed by the vertices $V_t, u_1, V_{13}, V_3, V_2, V_{23}.$ This means that $V_t$ is either adjacent to $V_{23}$ or to both $V_{13}$ and $V_2.$ If $V_t$ is adjacent to $V_2$ and $V_{13}$ we obtain the path $(V_2, V_t,u_1,V_{12},V_{23})$ thus $V_t$ must be adjacent to $V_{23}$ an analogously by symmetry, $V_t$ must be adjacent to $V_{12}.$ Hence, $V_t$ is adjacent to $V_{12},V_{23},V_1, V_3$ or $d(V_t)\geq 5$ in which case $V_{12}, V_{23} \in N(V_t)$ since otherwise we obtain the forbidden subgraph $G_6.$

     Notice that $V_2$ must be a complete subgraph since otherwise we obtain the forbidden subgraph $G_2$ formed by two non adjacent vertices of $V_2$ together with the vertices $V_3, u_1, V_{23}.$ The sets $V_{12}$ and $V_{23}$ must be independent since otherwise we obtain the forbidden subgraph $G_7$ formed by two adjacent vertices of $V_{12}$ together with the vertices $u_2,u_3,V_2,V_1$.

     If $V_t$ is not adjacent to $V_2$ nor to $V_{13}$, then $V_t$ must be a complete subgraph since otherwise we obtain the forbidden subgraph $G_6$ formed by two non adjacent vertices of $V_t$ together with $V_{13}, V_2, u_1,u_2$, and in this case $V_{13}$  must also be a complete subgraph since otherwise we obtain the forbidden subgraph $G_3$ formed by two non-adjacent vertices of $V_{13}$ together with $V_t, V_2,u_1,u_2.$
 \end{proof}

\subsection{Computing Allowable Graphs}\label{sec:allowable}
%{\color{red} A partir de aquí otra vez se usa la notación $12$ para el conjunto $V_{12}$ etc, en el texto es facil cambiarlo, pero no se en las figuras.}
%{\color{magenta} Ya cambie la notacion en las figuras.} {\color{violet} excelente, entonces lo cambio en el texto}

The above procedures can be accelerated using any programming language. 
Just by inputting the family $\mathcal{F}$ of 25 forbidden graphs (see Figure~\ref{fig:FG}) and searching among the graphs which are free of that set of graphs.
In this section we explain how the (algorithmic) procedure works and its results.

The idea is pretty simple and is the following: build every possible configuration of edges and check if that configuration contains a forbidden graph. 
If not, then it is considered as \texttt{allowable}. 
In order to accelerate the process, we build algorithms to run all combinations in parallel. 

Starting with a triangle, we labeled the set of these vertices as $u_1,u_2,u_3$.
With those vertices we considered the following set of vertices: the set of neighbors of vertex $u_i$, labeled them as $V_i$, the set of vertices adjacent to $u_i$ and $u_j$ and labeled them as $V_{ij}$ and finally, the set of vertices adjacent to the three vertices $u_1,u_2$, and $u_3$ and set it up as $V_t$. 
Of course, each of these sets of vertices may or may not exist and they (the vertices) may or may not be adjacent to other vertices. 
Whenever a set of vertices, say $V_1$, exist and adjacencies of it, it may or may not appear one of the forbidden graphs (see Figure \ref{fig:FG}) which tells us which non isomorphic \emph{configurations} from the initial triangle are admissible. 
The following is a pseudo-code that builds each configuration from the initial triangle and computes whether the configuration is allowed or not.

The procedure, ``IS-ALLOWABLE'', returns {\color{blue}\texttt{TRUE}} if the input $H$ contains any graph from the family $\mathcal{F}$ as an induced subgraph.
The second procedure, ``BUILD'', builds from the input set of desired vertices $N$, a graph built from a triangle and the vertices $N$ and creating the mandatory adjacencies (so for example if the vertex $V_{12}\in N$, then $V_{12}$ is adjacent to vertices $u_1$ and $u_2$. %{\color{red} are these vertices $u_{1}$ and $u_{2}$? Also, the notation for the vertices of the 3-cycle should be standardized.}).
The third procedure, ``POSSIBLE-EDGES'' returns the set of all possible edges that may occur within a configuration $G$ avoiding the not allowed edges in $G$ (as an example, $V_{1}$ can not be adjacent to vertex $u_2$).
For the second part of the algorithm, we get a family of allowable graphs when $k$ (fixed) vertices are added. 
To do so, the procedure ``LOCAL-FAMILY'' gets as an input a fix number $k$, gets all combinations of size $k$ of the set $\lbrace V_{12},V_{13},V_{23},V_1,V_2,V_3,V_t\rbrace$  and for each of those combinations, builds a graph $G$ using the procedure BUILD and then computes all possible combinations via the procedure POSSIBLE-EDGES to filter out which of those combinations are forbidden. 
Finally, the output is a set of all possible and allowed combinations when adding $k$ vertices to the triangle.
The last procedure, ``ALLOWABLE-GRAPHS'' just runs the procedure LOCAL-FAMILY for each $k$ and stores all allowable graphs. 
This results in a total of around 2000 graphs. 
However, many of these graphs share twin vertices.
After filtering these graphs out, we obtained a total of 449 non isomorphic allowable graphs.

\subsection{Blowing Up Allowable Graphs}

With the set of 449 graphs there is still one missing possibility. 
One could ``blow-up'' vertices into either complete graphs or independent sets with the rule that every new vertex is adjacent to all vertices that are originally adjacent to the blown-up vertex. 
The question is if this procedure can produce allowable graphs and if it could, what ``blown-up'' configurations are allowable. 
To check complete graphs, a vertex is replaced with a set of two non adjacent vertices and test if the new graph is no longer allowable. 
Similarly, to check allowable independent sets, a vertex is replaced with an edge and test if the new graph is no longer allowable. 
Nevertheless, it could be the case that independent sets and complete graphs keep the graph allowable. 
The next step is to search for more forbidden graphs within the process of blowing up vertices. 
By Theorem~\ref{thm_mz_fin}, we just have to search within the set of graphs of size eight. 
It turns out that the only vertex that in some cases admits complete graphs and independent sets is the vertex $V_t$.

\bigskip

\begin{lstlisting}[mathescape=true, basicstyle=\ttfamily, numbers=left, commentstyle=\itshape]
Procedure IS-ALLOWABLE($H$)  // Whether $H$ contains a forbidden induced subgraph
    $\mathcal{F} \gets$ Family of forbidden graphs \ref{fig:FG}
    For $G \in \mathcal{F}$
        If $G$ is an induced subgraph of $H$
            return FALSE
        EndIf
    EndFor
    return TRUE
EndProcedure

Procedure BUILD($N$)  // Adds vertices from $N$ and its adjacencies
    Let $G$ be the triangle with vertices $u_{1}, u_{2}$, and $u_{3}$.
    $V(G) \gets V(G) \cup N$.
    If $V_{12} \in N$
        $E(G) \gets E(G) \cup \{(V_{12},u_{1}), (V_{12},u_{2})\}$
    EndIf
    If $V_{13} \in N$
        $E(G) \gets E(G) \cup \{(V_{13},u_{1}), (V_{13},u_{3})\}$
    EndIf
    If $V_{23} \in N$
        $E(G) \gets E(G) \cup \{(V_{23},u_{2}), (V_{23},u_{3})\}$
    EndIf
    If $V_1 \in N$
        $E(G) \gets E(G) \cup \{(V_1,u_{1})\}$
    EndIf
    If $V_2 \in N$
        $E(G) \gets E(G) \cup \{(V_2,u_{2})\}$
    EndIf
    If $V_3 \in N$
        $E(G) \gets E(G) \cup \{(V_3,u_{3})\}$
    EndIf
    If $V_t \in N$
        $E(G) \gets E(G) \cup \{(V_t,u_{1}), (V_t,u_{2}), (V_t,u_{3})\}$
    EndIf
    return $G$.
EndProcedure

Procedure POSSIBLE-EDGES($G$)  // From all possible edges, there are some that are not allowed
    $F \gets \{(V_1,u_{2}), (V_1,u_{3}), (V_2,u_{1}), (V_2,u_{3}), (V_3,u_{1}), (V_3,u_{2}), (V_{12},u_{3}), (V_{13},u_{2}), (V_{23},u_{1})\}$
    $E \gets E(G)$  // Set of edges of $E$
    $K \gets E(K_{|G|})$  // Set of all possible edges of a graph of the size of $G$
    return $K \setminus (E \cup F)$.
EndProcedure
\end{lstlisting}

\begin{lstlisting}[mathescape=true, basicstyle=\ttfamily, numbers=left, commentstyle=\itshape, caption={Allowable Local Families}, label={alg:allowable1}]
Procedure LOCAL-FAMILY($k$)  // Returns the admissible graphs when $k$ vertices are added to the triangle.
    $V \gets \{V_{12}, V_{13}, V_{23}, V_1, V_2, V_3, V_t\}$
    $L \gets \emptyset$  // Local family set
    $\mathcal{V} \gets \{$All combinations of size $k$ of $V\}$
    For $N \in \mathcal{V}$
        $G \gets$ BUILD($N$)
        $E \gets$ POSSIBLE-EDGES($G$)
        $m = |E|$
        For $0 \leq s \leq m$
            $\mathcal{E} \gets \{$All combinations of size $s$ of $E\}$  // Consider each combination of size $s$ from all $m$ possible edges that can be added.
            For $E' \in \mathcal{E}$
                $H \gets G$  // A copy of $G$
                $E(H) \gets E(H) \cup E'$
                If IS-ALLOWABLE($H$)
                    $L \gets L \cup H$
                EndIf
            EndFor
        EndFor
    EndFor
    return $L$
EndProcedure
\end{lstlisting}

\begin{lstlisting}[mathescape=true, basicstyle=\ttfamily, numbers=left, commentstyle=\itshape, caption={Allowable Graphs}]
Procedure ALLOWABLE-GRAPHS()
    $L \gets \emptyset$
    For $0 \leq k \leq 7$
        $L' \gets$ LOCAL-FAMILY($k$)
        $L \gets L' \cup L$
    EndFor
    return $L$
EndProcedure
\end{lstlisting}

\section*{Acknowledgments}
C.A. Alfaro was partially supported by SNI and CONACyT.  Jephian C.-H. Lin was partially supported by NSTC grant 113-2115-M-110-010-MY3. Teresa I. Hoekstra-Mendoza was partially supported by SNI 829061. Juan Pablo Serrano was supported by CONAHCyT, grant number 1174526.

\appendix

\section{Allowed Families}

In this section we present the families of allowed graphs with girth three and mz at most 3.

A magenta vertex represents a clique, a yellow vertex is an independent set, a blue vertex is a graph free of $\lbrace K_2\cup K_1\cup K_1,P_3\cup K_1, P_4\rbrace$ and a green vertex represents a graph free of the 11 graphs shown in~\ref{fig:greenvertex}.

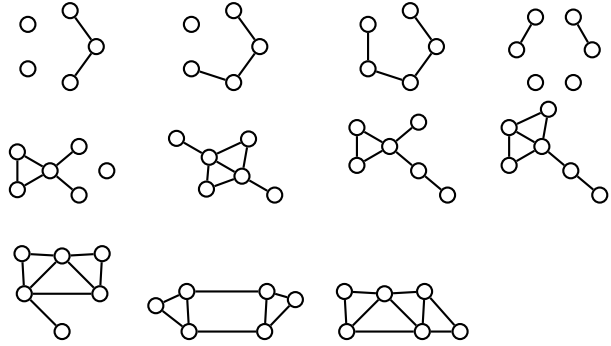
\begin{figure}[h]
\centering
\begin{tabular}{cccc}
\begin{tikzpicture}[scale=0.5,thick]
\tikzstyle{every node}=[minimum width=0pt, inner sep=2pt, circle]
    \draw (0:1) node[draw] (0) {};
    \draw (72:1) node[draw] (2) {};
    \draw (144:1) node[draw] (3) {};
    \draw (216:1) node[draw] (1) {};
    \draw (288:1) node[draw] (4) {};
    \draw  (0) edge (2);
    \draw  (0) edge (4);
\end{tikzpicture}
&
\begin{tikzpicture}[scale=0.5,thick]
\tikzstyle{every node}=[minimum width=0pt, inner sep=2pt, circle]
    \draw (0:1) node[draw] (0) {};
    \draw (72:1) node[draw] (2) {};
    \draw (144:1) node[draw] (3) {};
    \draw (216:1) node[draw] (1) {};
    \draw (288:1) node[draw] (4) {};
    \draw  (0) edge (2);
    \draw  (0) edge (4);
    \draw  (1) edge (4);
\end{tikzpicture}
&
\begin{tikzpicture}[scale=0.5,thick]
\tikzstyle{every node}=[minimum width=0pt, inner sep=2pt, circle]
    \draw (0:1) node[draw] (0) {};
    \draw (72:1) node[draw] (2) {};
    \draw (144:1) node[draw] (3) {};
    \draw (216:1) node[draw] (1) {};
    \draw (288:1) node[draw] (4) {};
    \draw  (0) edge (2);
    \draw  (0) edge (4);
    \draw  (1) edge (4);
    \draw  (1) edge (3);
\end{tikzpicture}
&
\begin{tikzpicture}[scale=0.5,thick]
\tikzstyle{every node}=[minimum width=0pt, inner sep=2pt, circle]
    \draw (0:1) node[draw] (0) {};
    \draw (60:1) node[draw] (1) {};
    \draw (120:1) node[draw] (2) {};
    \draw (180:1) node[draw] (3) {};
    \draw (240:1) node[draw] (4) {};
    \draw (300:1) node[draw] (5) {};
    
    \draw  (0) edge (1);
    \draw  (2) edge (3);
\end{tikzpicture}\\

%$\tiny P_3\cup K_1\cup K_1$&$\tiny P_4\cup K_1$&$\tiny P_5$&$\tiny K_2\cup K_2\cup K_1\cup K_1$\\\\

\begin{tikzpicture}[scale=0.5,thick]
\tikzstyle{every node}=[minimum width=0pt, inner sep=2pt, circle]
    \draw (0:1.5) node[draw] (0) {};
    \draw (40:1) node[draw] (1) {};
    \draw (150:1) node[draw] (2) {};
    \draw (210:1) node[draw] (3) {};
    \draw (0:0) node[draw] (4) {};
    \draw (320:1) node[draw] (5) {};
    
    \draw  (1) edge (4);
    \draw  (2) edge (4);
    \draw  (2) edge (3);
    \draw  (3) edge (4);
    \draw  (4) edge (5);
\end{tikzpicture}
&
\begin{tikzpicture}[scale=0.5,thick]
\tikzstyle{every node}=[minimum width=0pt, inner sep=2pt, circle]
    \draw (200:1) node[draw] (0) {};
    \draw (80:1) node[draw] (1) {};
    \draw (150:2) node[draw] (2) {};
    \draw (150:1) node[draw] (3) {};
    \draw (0:0) node[draw] (4) {};
    \draw (330:1) node[draw] (5) {};
    
    \draw  (1) edge (4);
    \draw  (1) edge (3);
    \draw  (0) edge (4);
    \draw  (0) edge (3);
    \draw  (2) edge (3);
    \draw  (3) edge (4);
    \draw  (4) edge (5);
\end{tikzpicture}
&
\begin{tikzpicture}[scale=0.5,thick]
\tikzstyle{every node}=[minimum width=0pt, inner sep=2pt, circle]
    \draw (320:2) node[draw] (0) {};
    \draw (40:1) node[draw] (1) {};
    \draw (150:1) node[draw] (2) {};
    \draw (210:1) node[draw] (3) {};
    \draw (0:0) node[draw] (4) {};
    \draw (320:1) node[draw] (5) {};
    
    \draw  (1) edge (4);
    \draw  (2) edge (4);
    \draw  (2) edge (3);
    \draw  (3) edge (4);
    \draw  (4) edge (5);
    \draw  (0) edge (5);
\end{tikzpicture}
&
\begin{tikzpicture}[scale=0.5,thick]
\tikzstyle{every node}=[minimum width=0pt, inner sep=2pt, circle]
    \draw (320:2) node[draw] (0) {};
    \draw (80:1) node[draw] (1) {};
    \draw (150:1) node[draw] (2) {};
    \draw (210:1) node[draw] (3) {};
    \draw (0:0) node[draw] (4) {};
    \draw (320:1) node[draw] (5) {};
    
    \draw  (1) edge (4);
    \draw  (1) edge (2);
    \draw  (2) edge (4);
    \draw  (2) edge (3);
    \draw  (3) edge (4);
    \draw  (4) edge (5);
    \draw  (0) edge (5);
\end{tikzpicture}
\\\\
\begin{tikzpicture}[scale=0.5,thick]
\tikzstyle{every node}=[minimum width=0pt, inner sep=2pt, circle]
    \draw (90:1) node[draw] (0) {};
    \draw (180:1) node[draw] (1) {};
    \draw (0:1) node[draw] (2) {};
    \draw (135:1.5) node[draw] (3) {};
    \draw (45:1.5) node[draw] (4) {};
    \draw (270:1) node[draw] (5) {};
    
    \draw  (0) edge (1);
    \draw  (0) edge (2);
    \draw  (1) edge (2);
    \draw  (1) edge (3);
    \draw  (2) edge (4);
    \draw  (0) edge (4);
    \draw  (0) edge (3);
    \draw  (1) edge (5);
\end{tikzpicture}
&
\begin{tikzpicture}[scale=0.5,thick]
\tikzstyle{every node}=[minimum width=0pt, inner sep=2pt, circle]
    \draw (180:1) node[draw] (0) {};
    \draw (135:1.5) node[draw] (1) {};
    \draw (45:1.5) node[draw] (2) {};
    \draw (0:1) node[draw] (3) {};
    \draw (160:2) node[draw] (4) {};
    \draw (25:2) node[draw] (5) {};
    
    \draw  (0) edge (1);
    \draw  (1) edge (2);
    \draw  (2) edge (3);
    \draw  (0) edge (3);
    \draw  (0) edge (4);
    \draw  (1) edge (4);
    \draw  (2) edge (5);
    \draw  (3) edge (5);
\end{tikzpicture}
&
\begin{tikzpicture}[scale=0.5,thick]
\tikzstyle{every node}=[minimum width=0pt, inner sep=2pt, circle]
    \draw (90:1) node[draw] (0) {};
    \draw (180:1) node[draw] (1) {};
    \draw (0:1) node[draw] (2) {};
    \draw (135:1.5) node[draw] (3) {};
    \draw (45:1.5) node[draw] (4) {};
    \draw (0:2) node[draw] (5) {};
    
    \draw  (0) edge (1);
    \draw  (0) edge (2);
    \draw  (1) edge (2);
    \draw  (1) edge (3);
    \draw  (2) edge (4);
    \draw  (0) edge (4);
    \draw  (0) edge (3);
    \draw  (2) edge (5);
    \draw  (4) edge (5);
\end{tikzpicture}
\end{tabular}
\caption{Minimal forbidden graphs for a green vertex.}
\label{fig:greenvertex}
\end{figure}

\subsection{A repository}
A file containing the 449 graphs with its forbidden graphs per vertex can be found in the link 

\href{https://github.com/serranoperez/math/blob/main/forbidden_lists.pkl}{https://github.com/serranoperez/math/blob/main/forbidden\_lists.pkl}

\noindent It is a key-value structure with 449 keys and each value is another key-value structure with two keys: ``graph'' and ``forbidden\_list''. The key ``graphs'' contains a SageMathCell version of the allowed graph and ``forbidden\_list'' is another key-value structure with the vertices $12,13,23,V_1,V_2,V_3$, and $V_t$ as keys and each value of the key is a list of the forbidden graphs for that vertex.  
A graph in this file was encoded using the \texttt{Base64} method.
To decode a the graph use the command \texttt{base64.b64decode({\color{blue} string})}. 
To obtain the Graph object in Sage use the command \texttt{pickle.load(base64.b64decode({\color{blue} string}))}

\input{show}

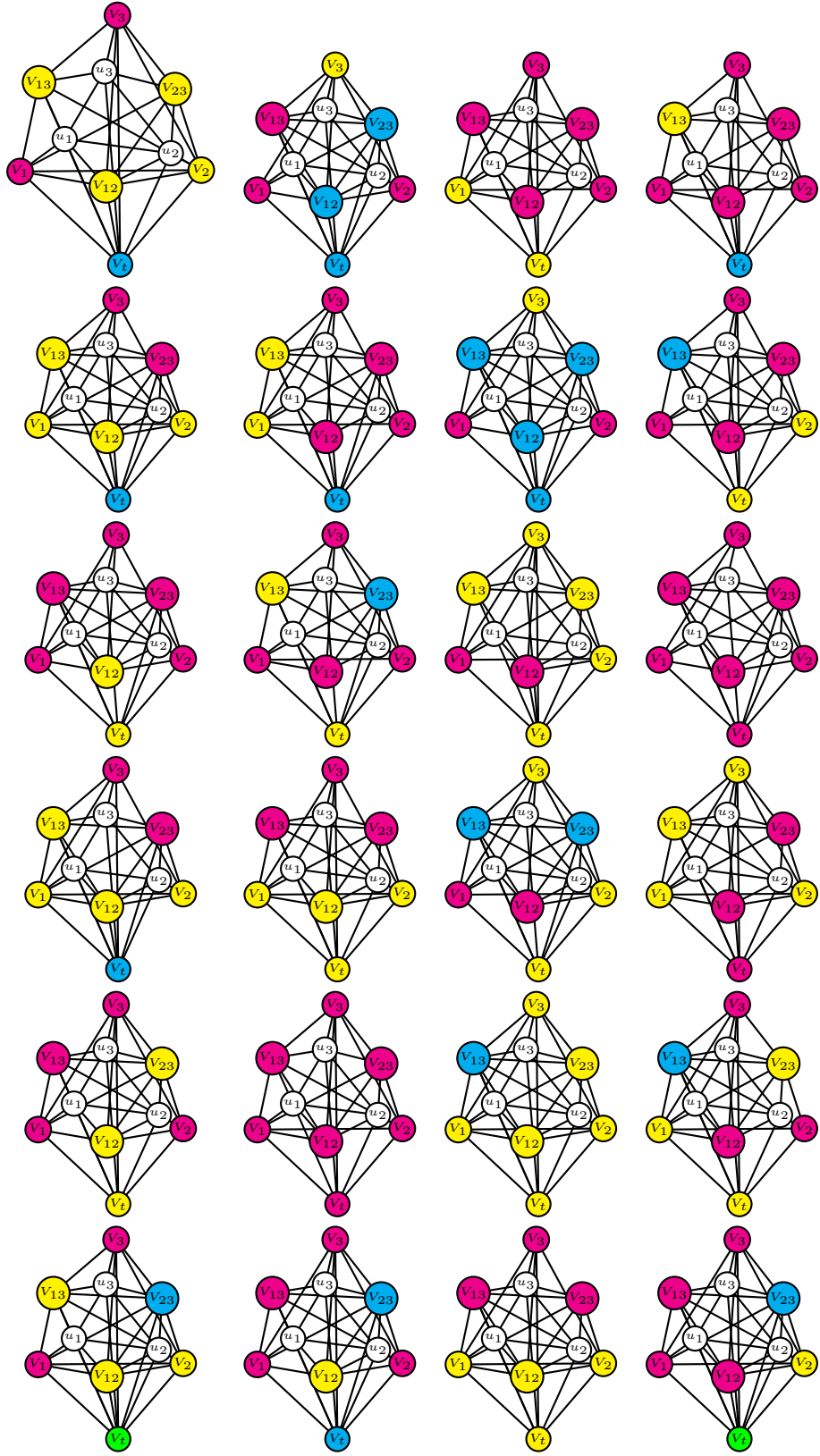
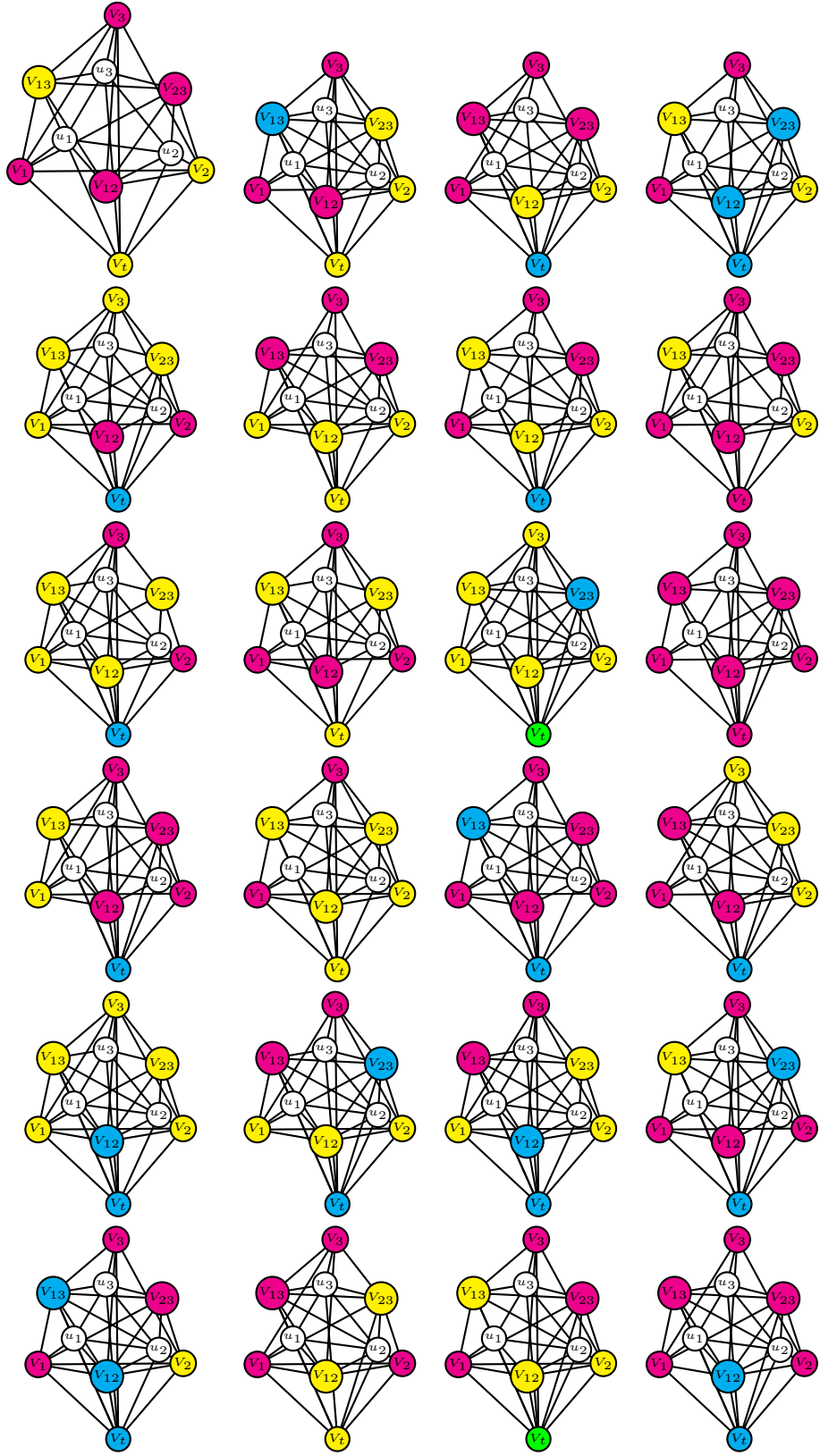
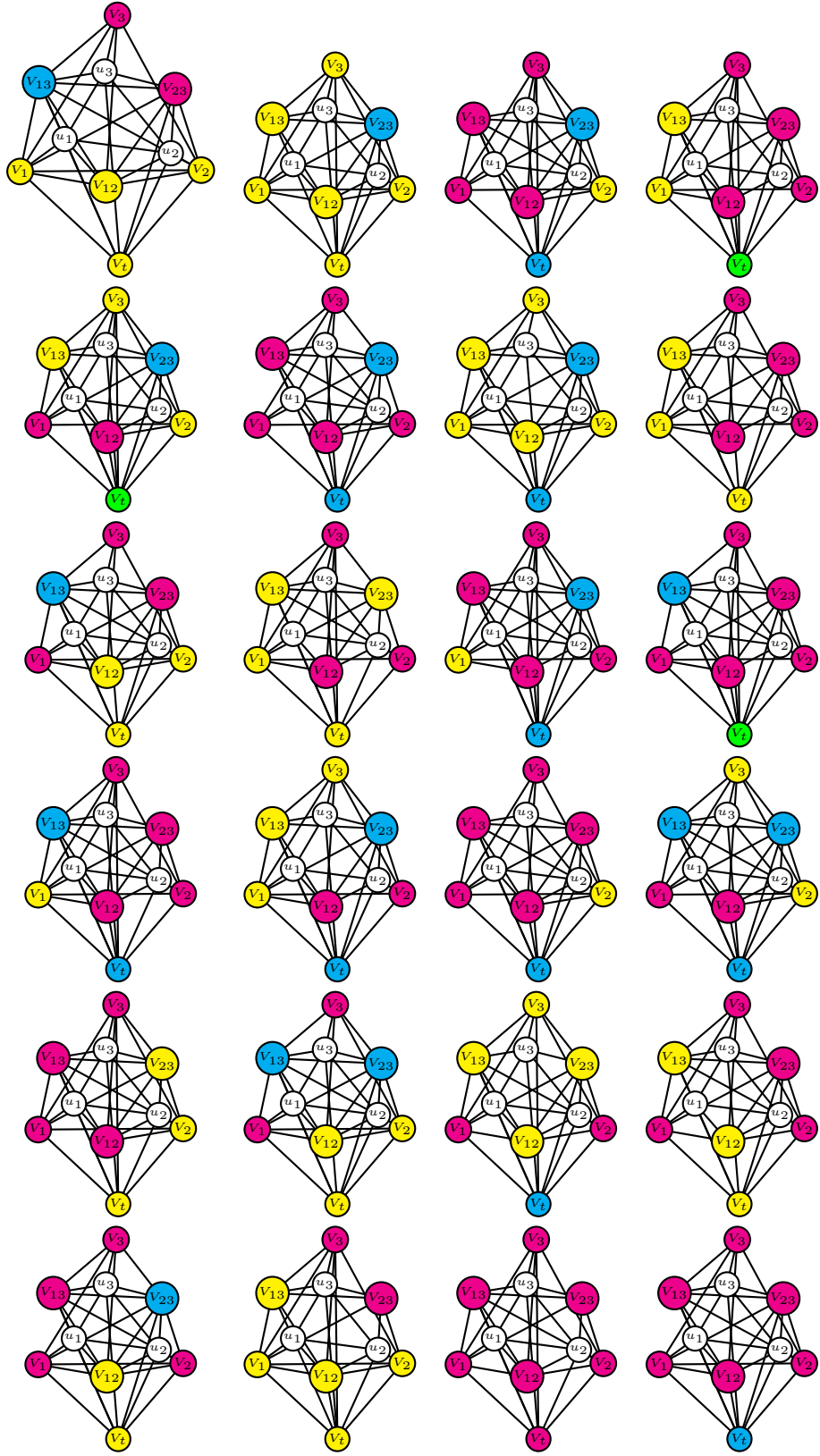
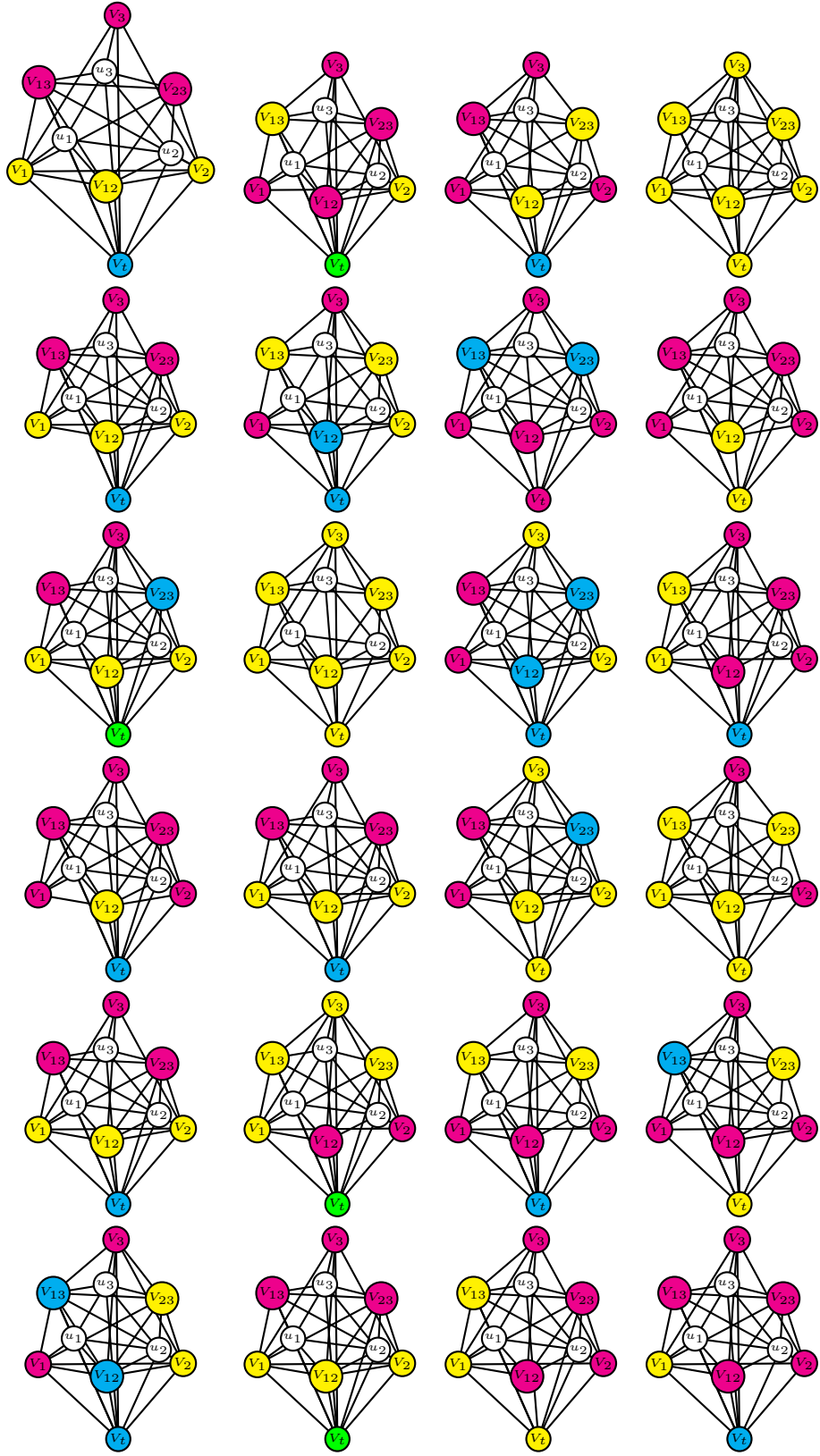
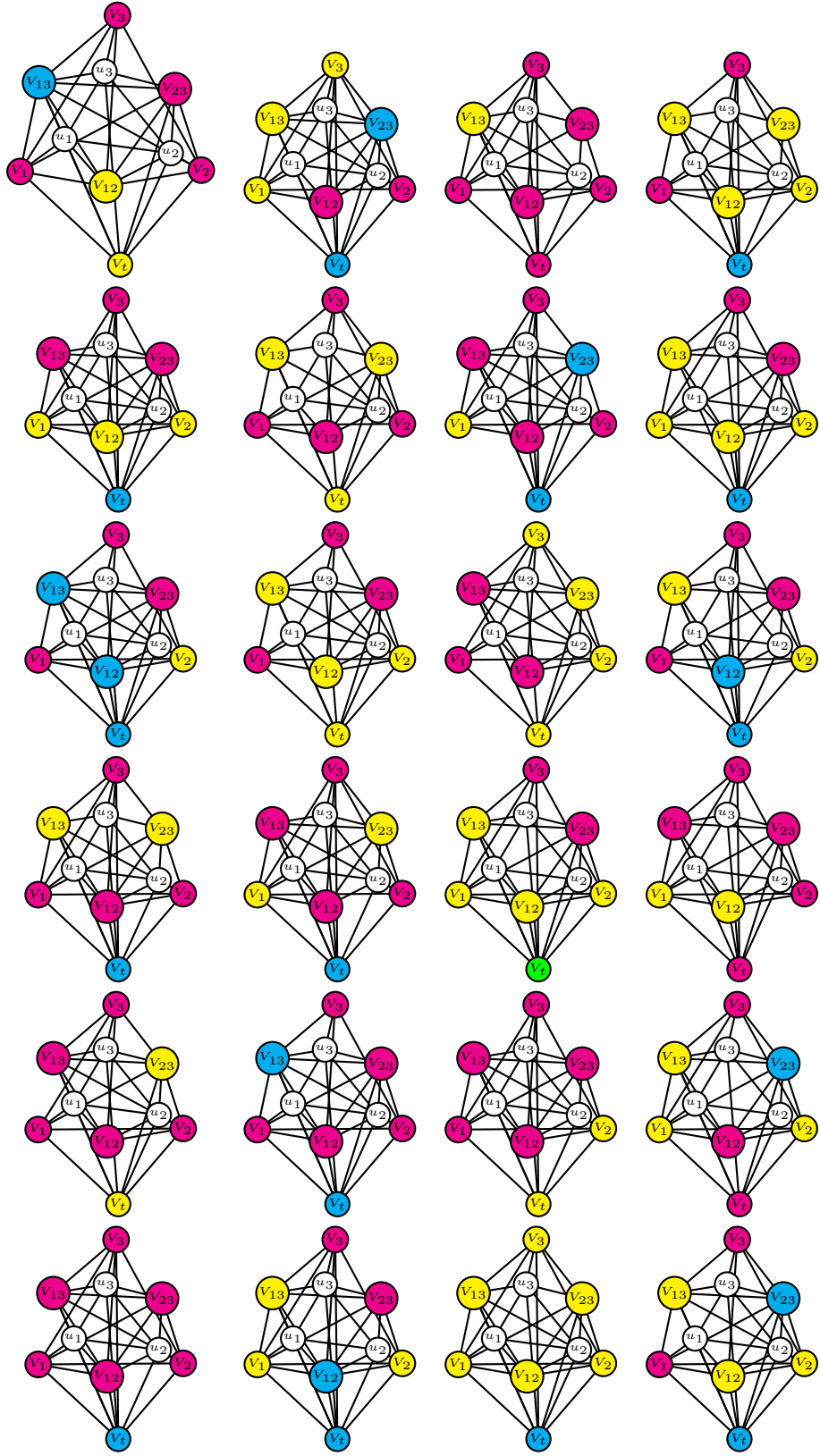
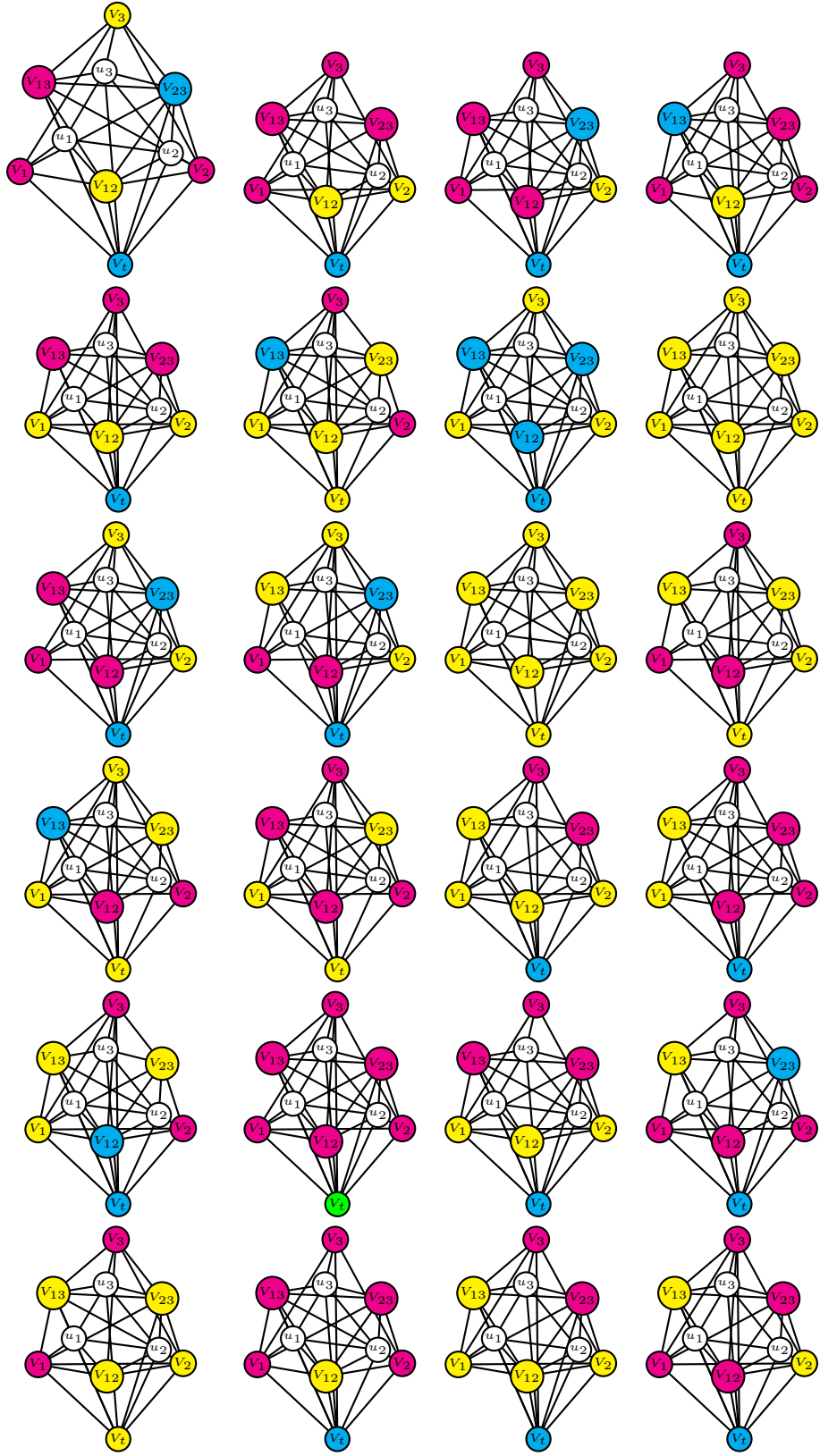
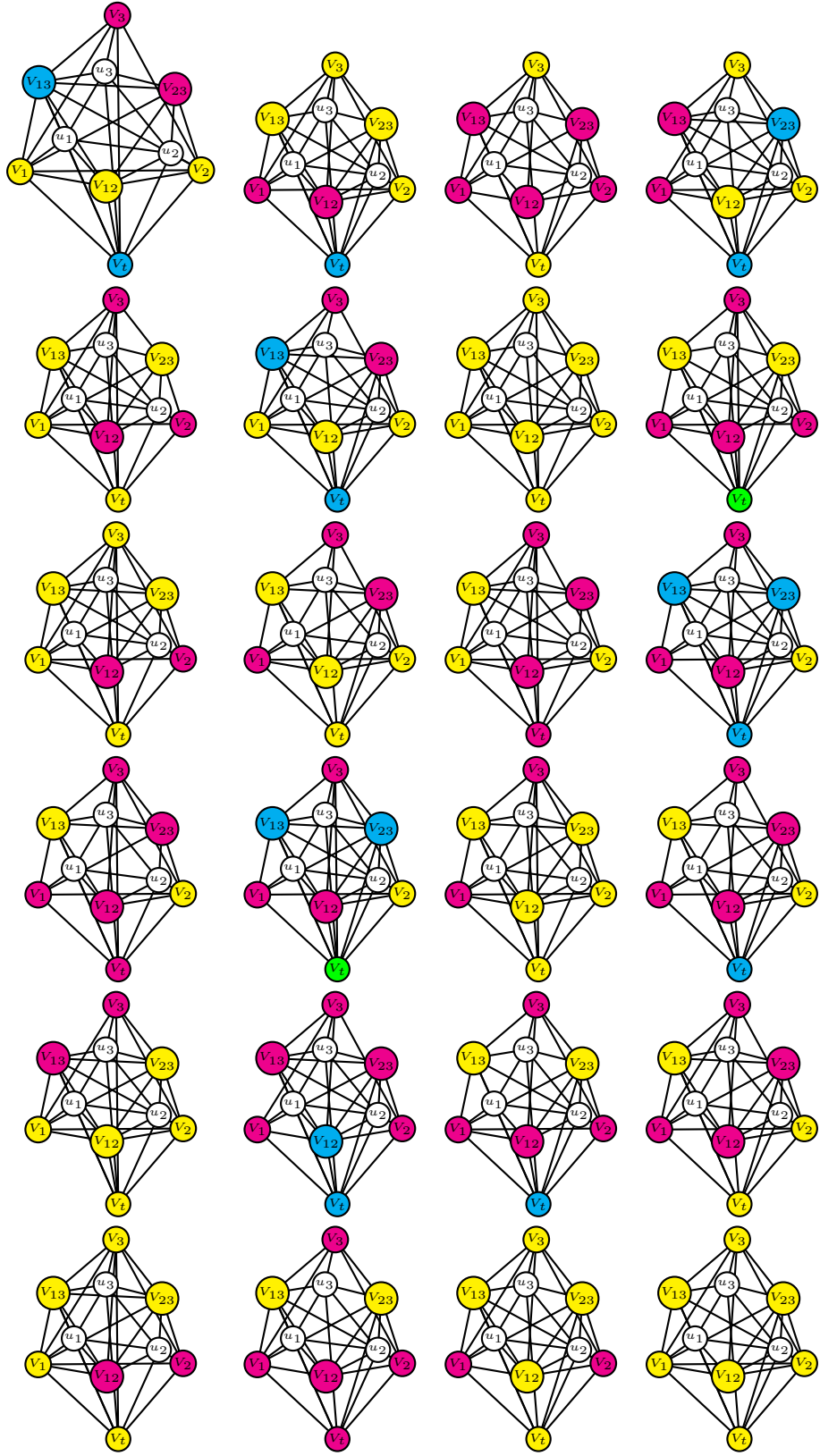
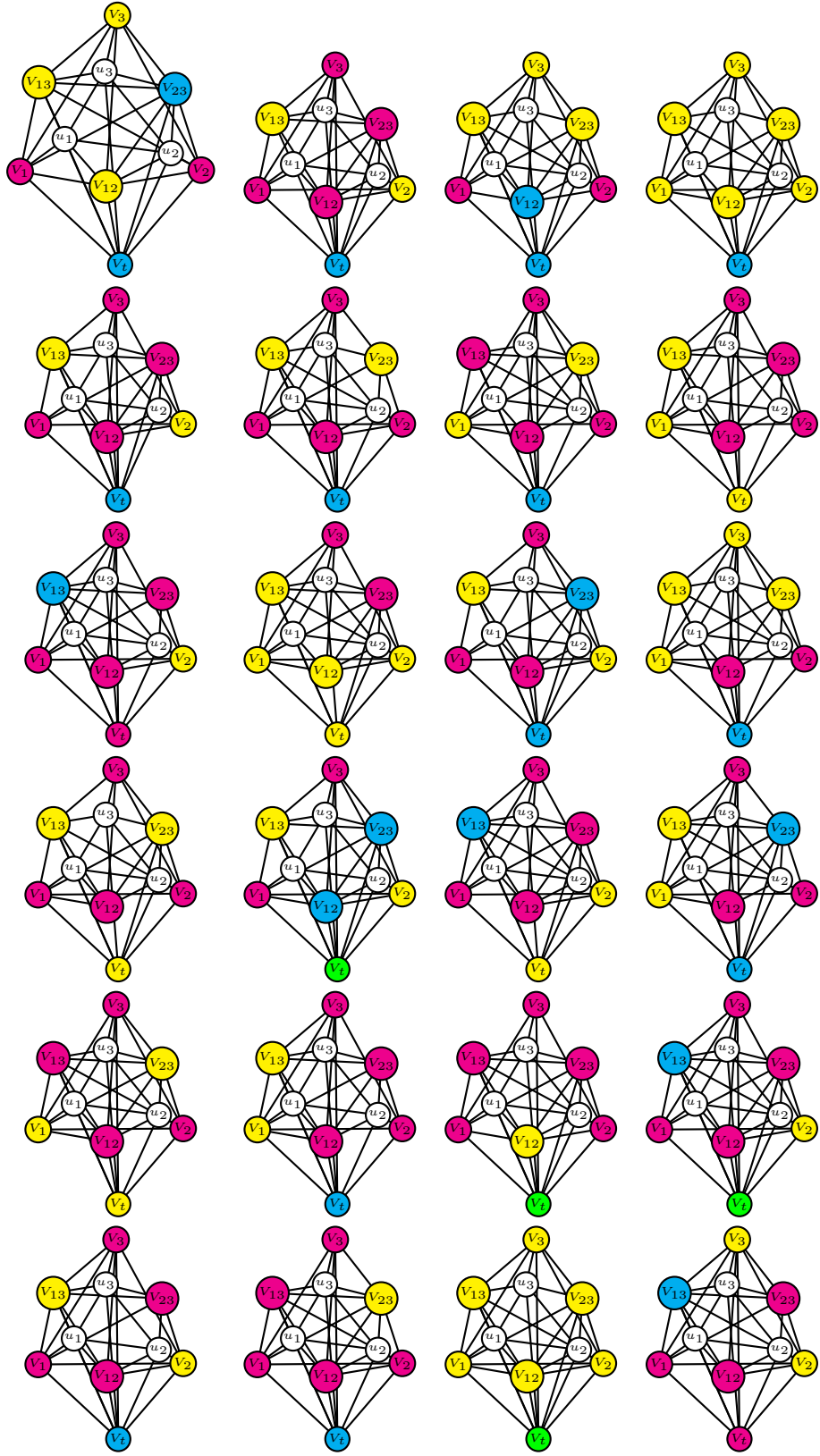
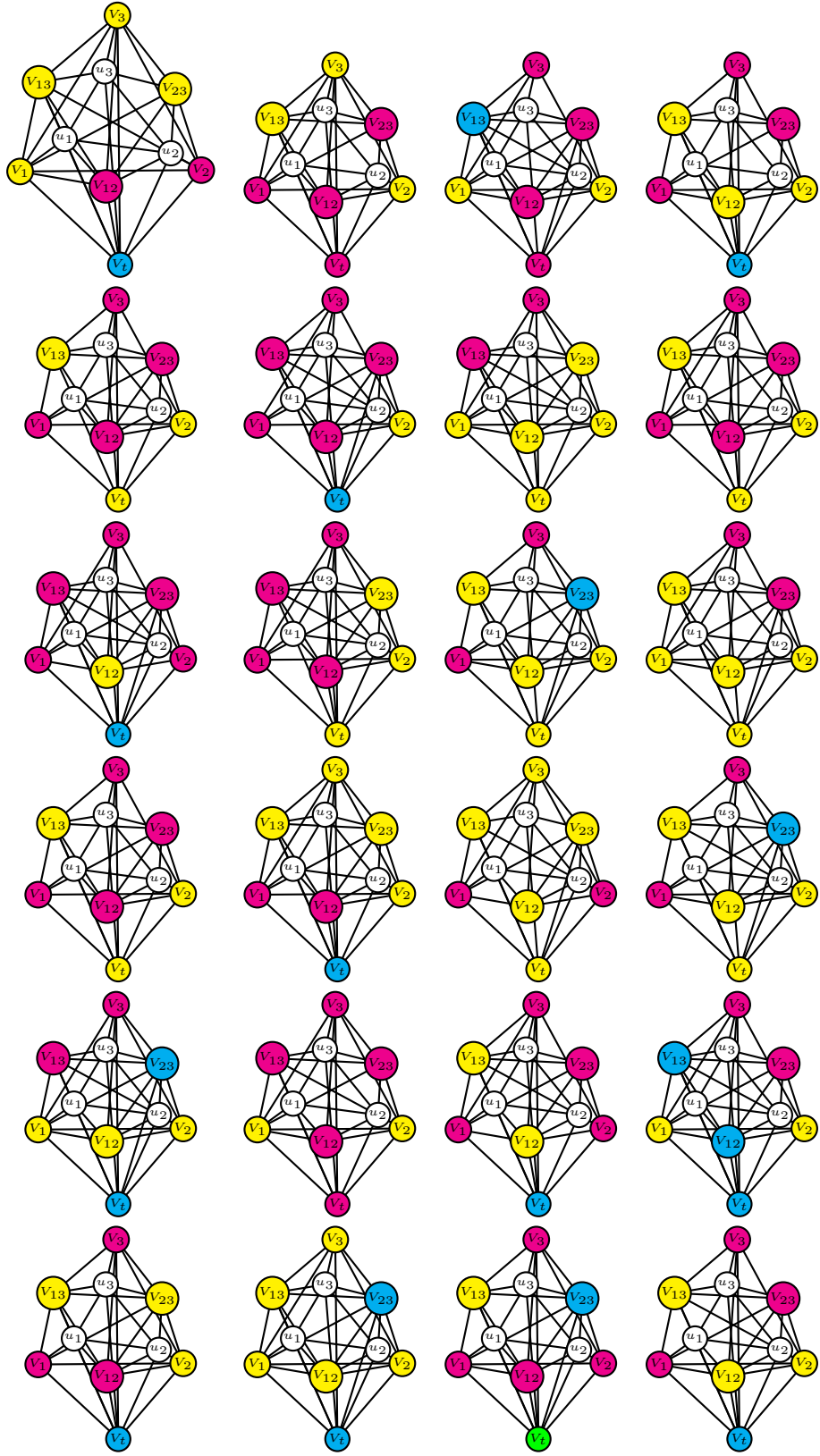
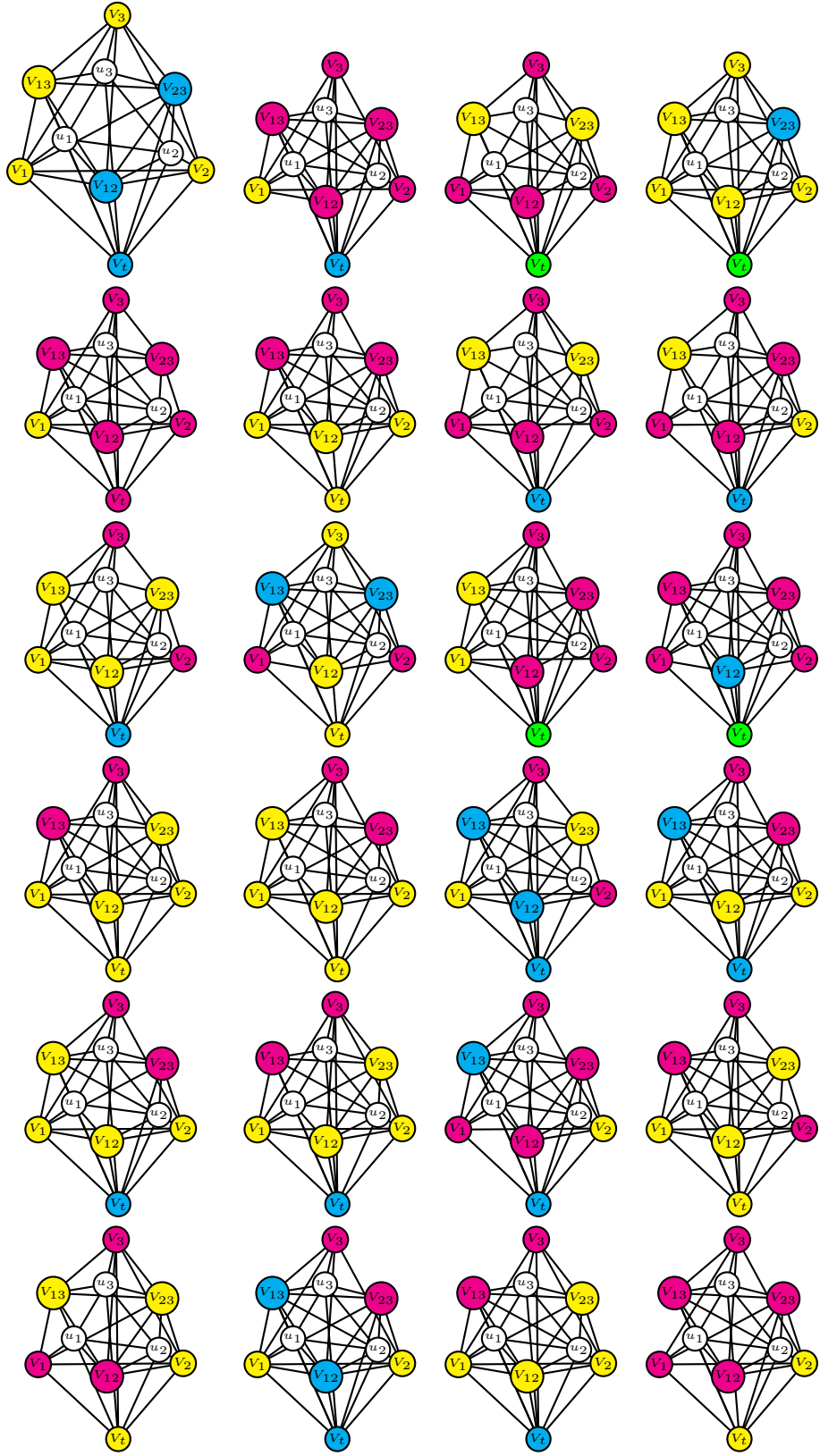
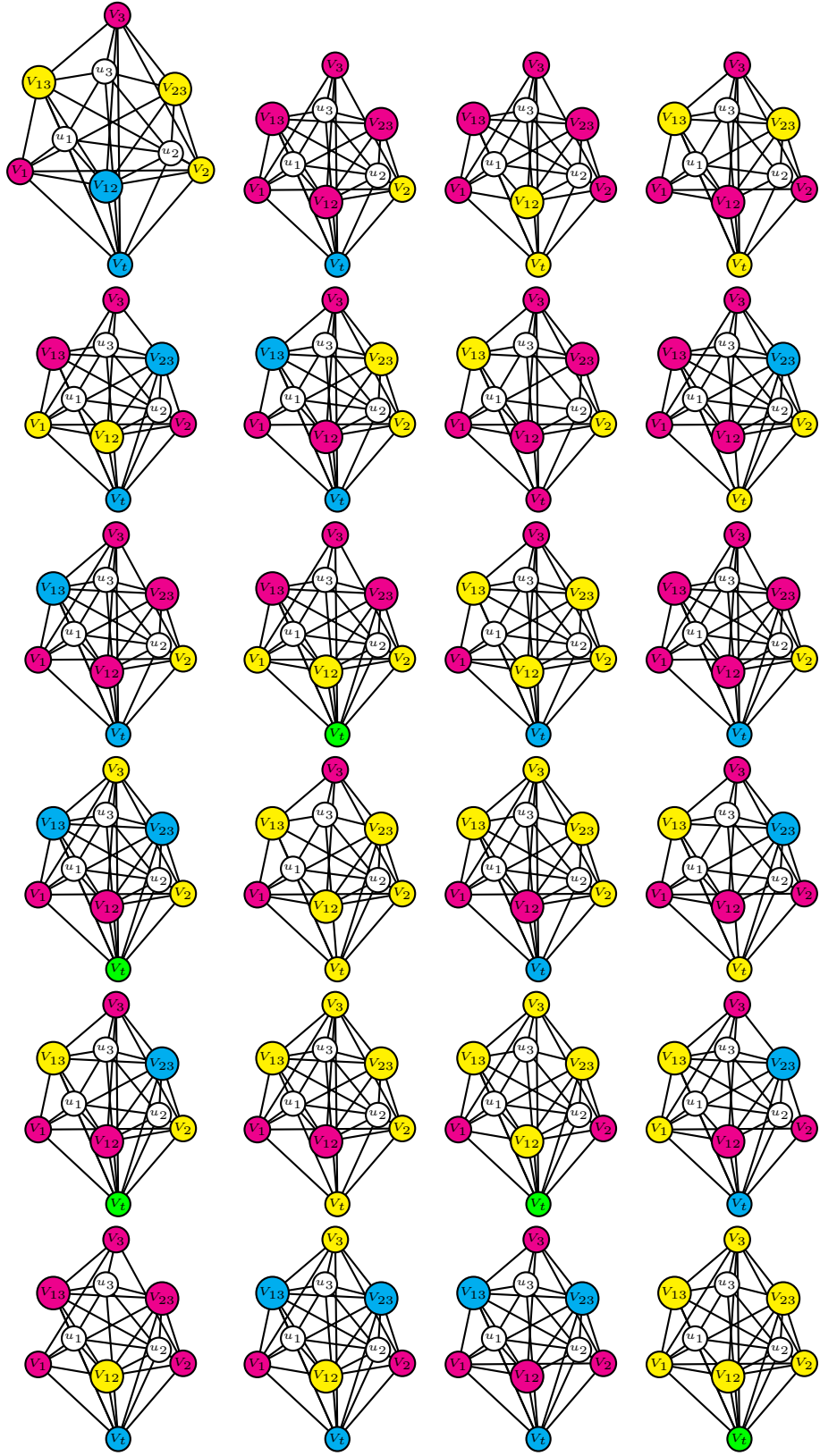
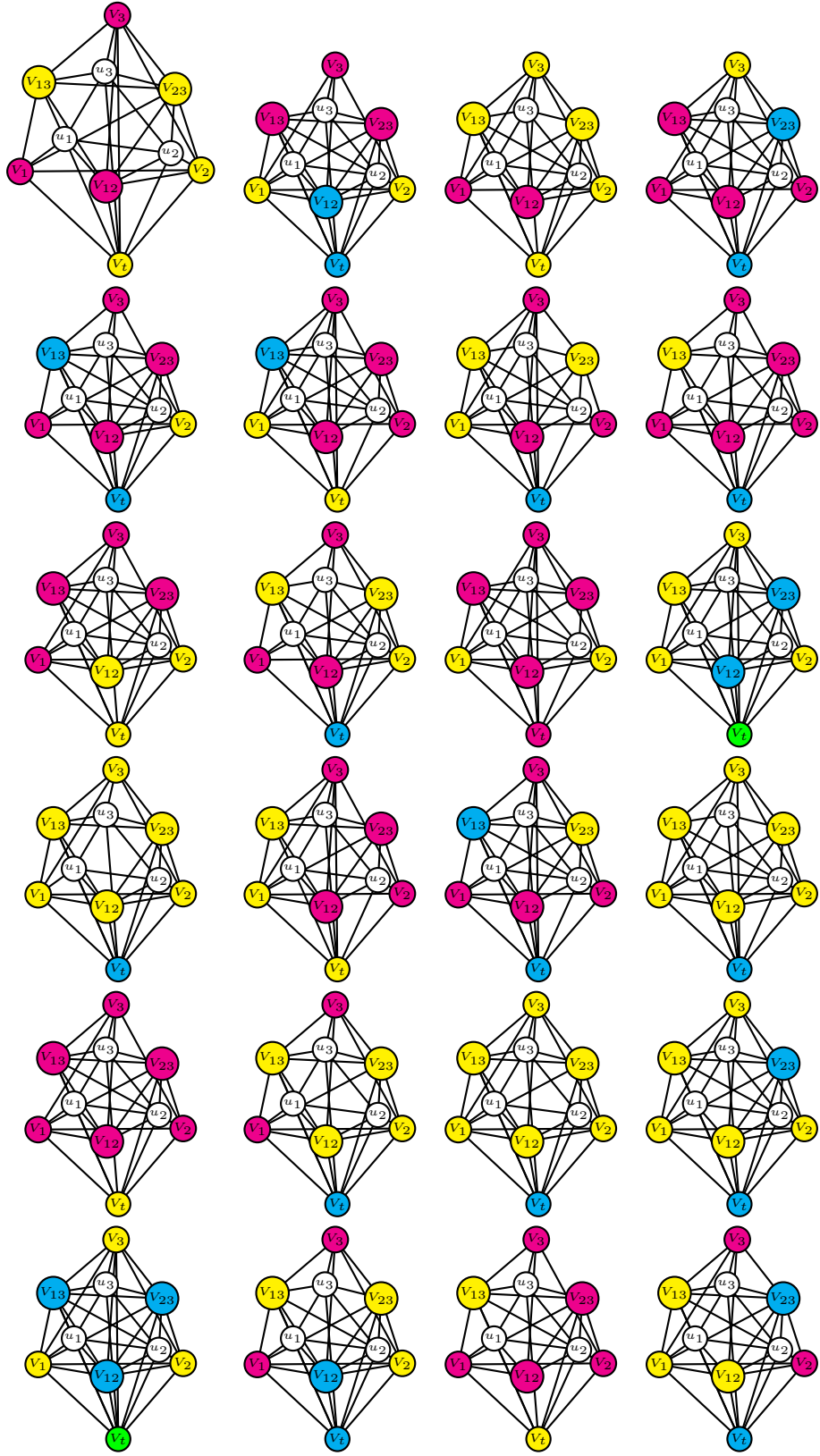
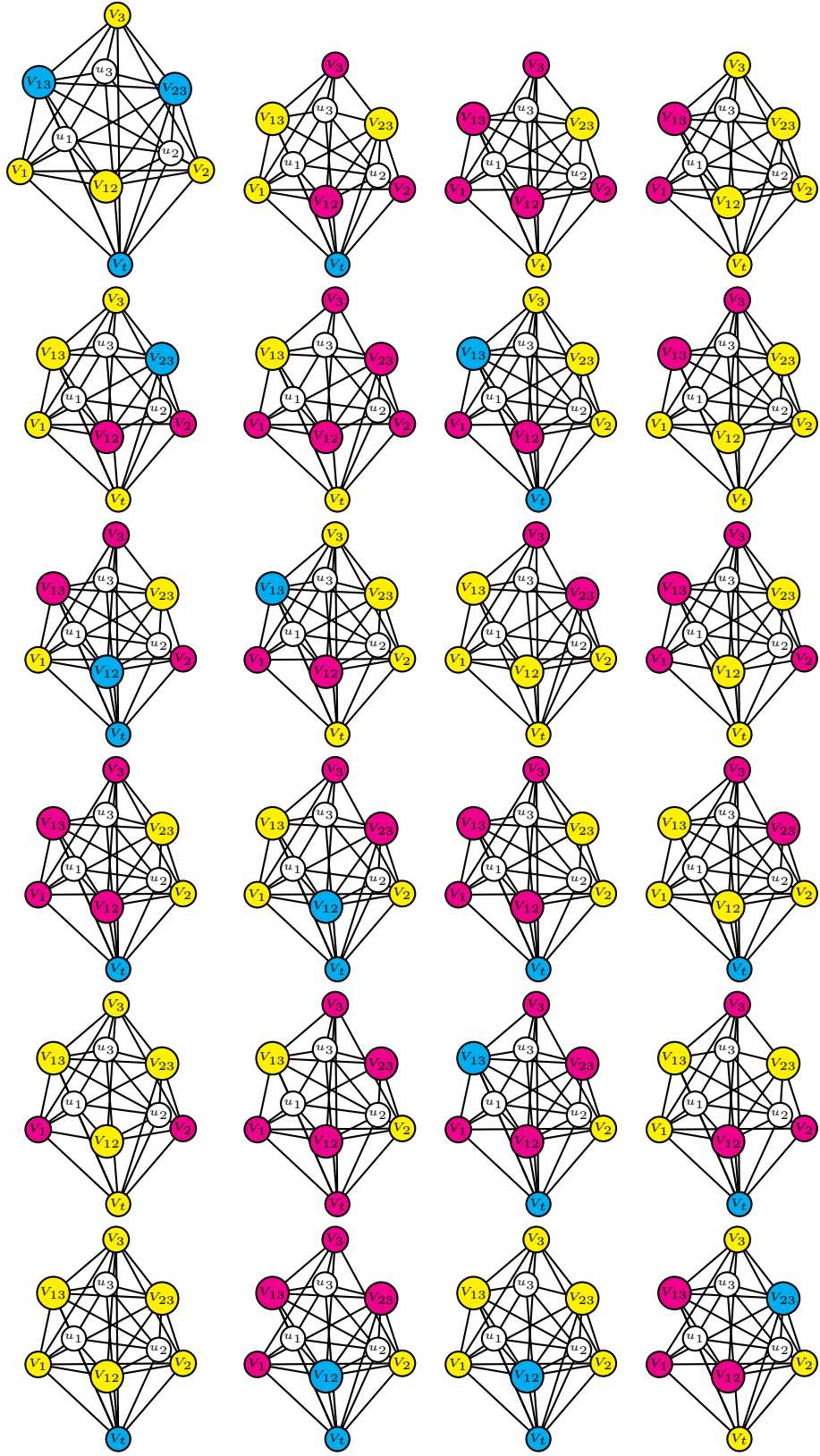
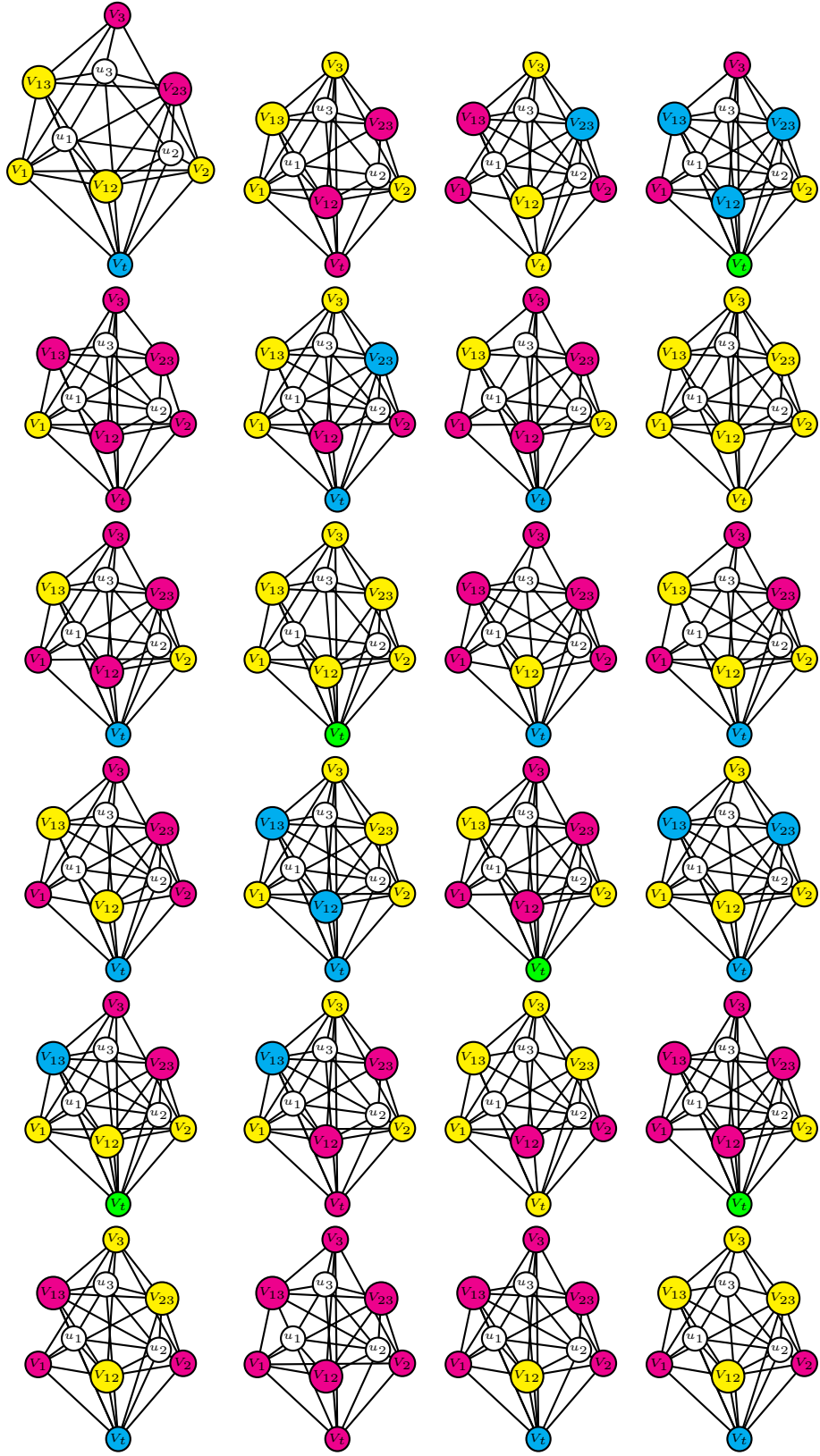
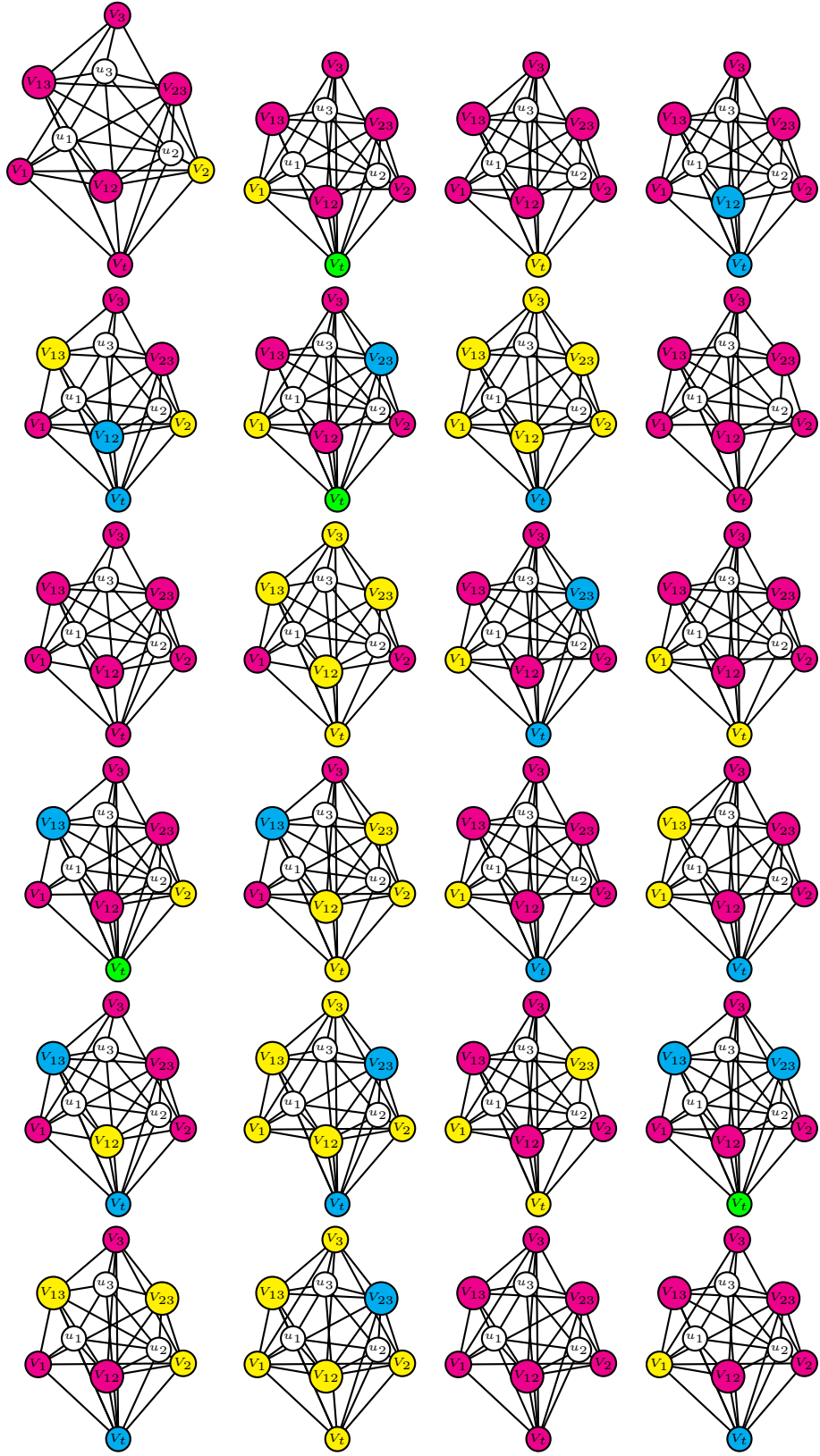
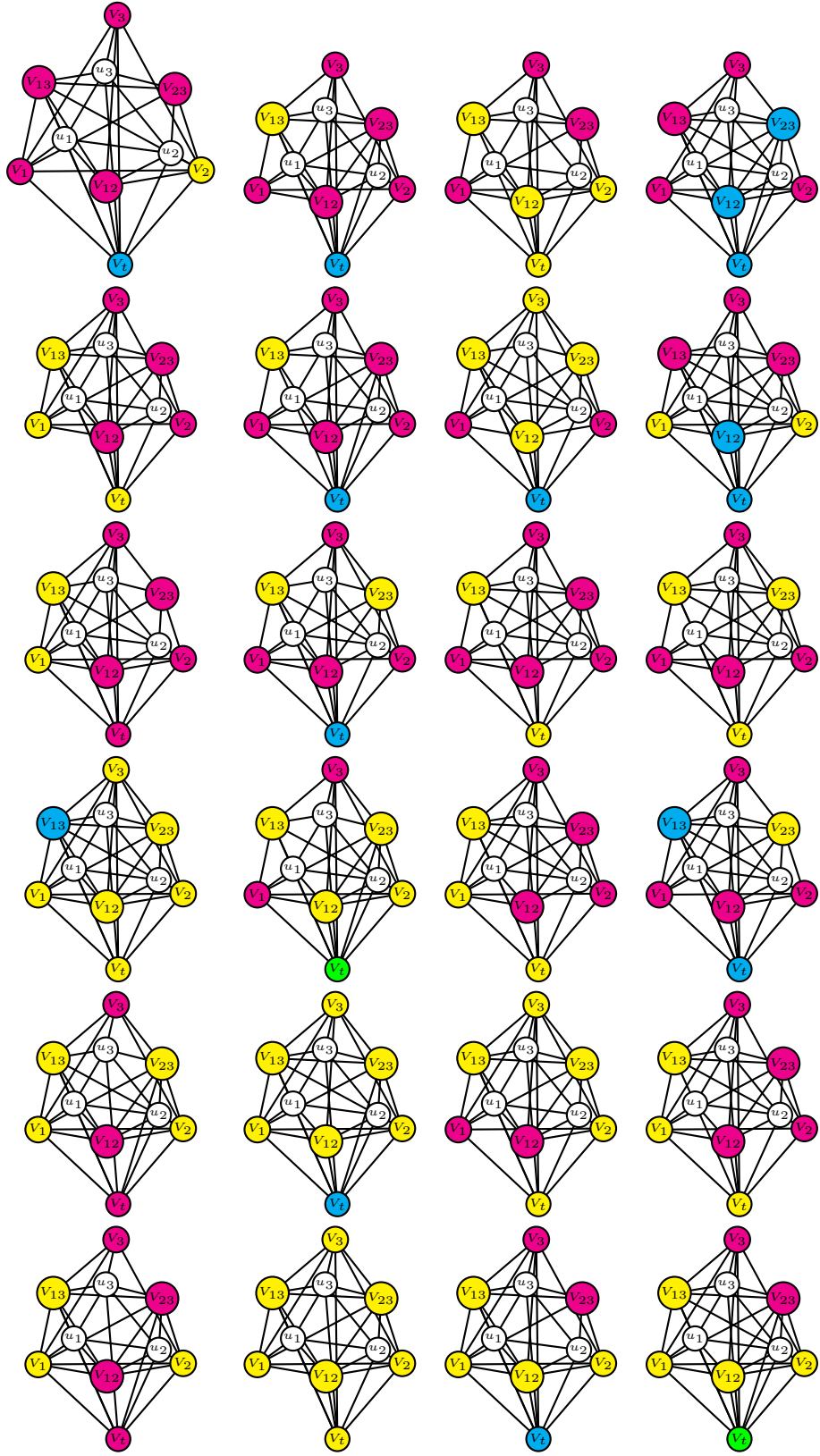
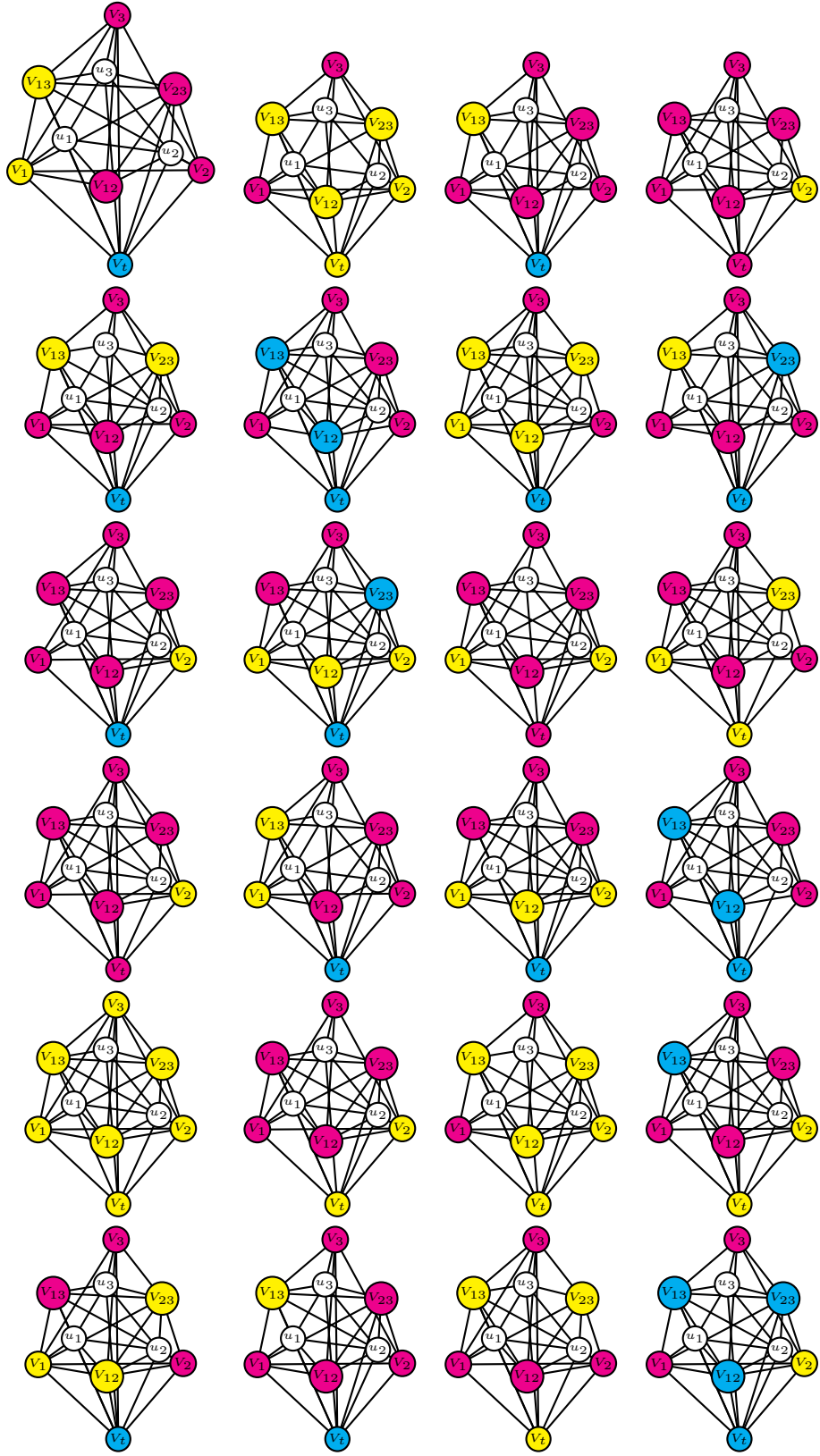
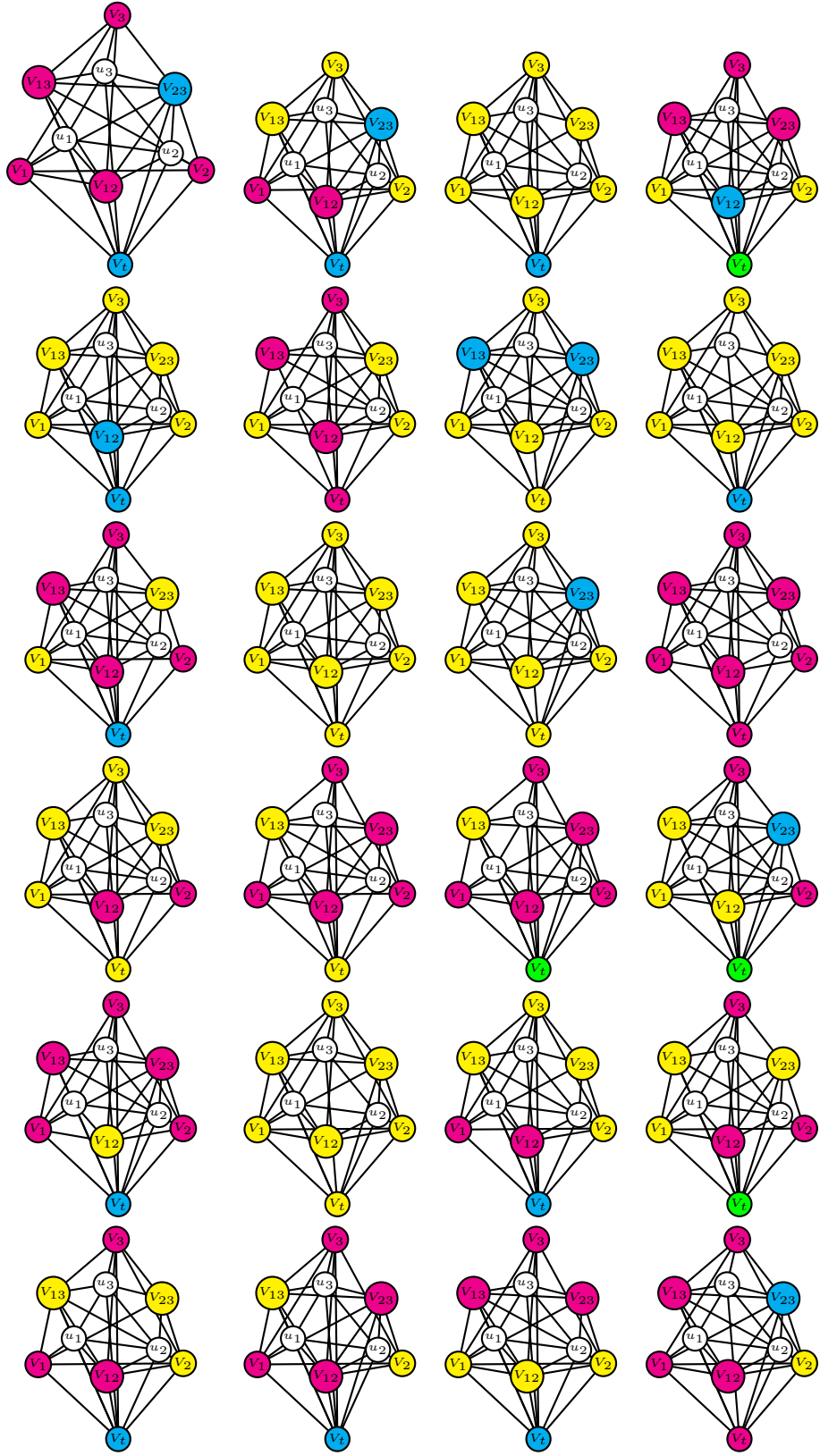
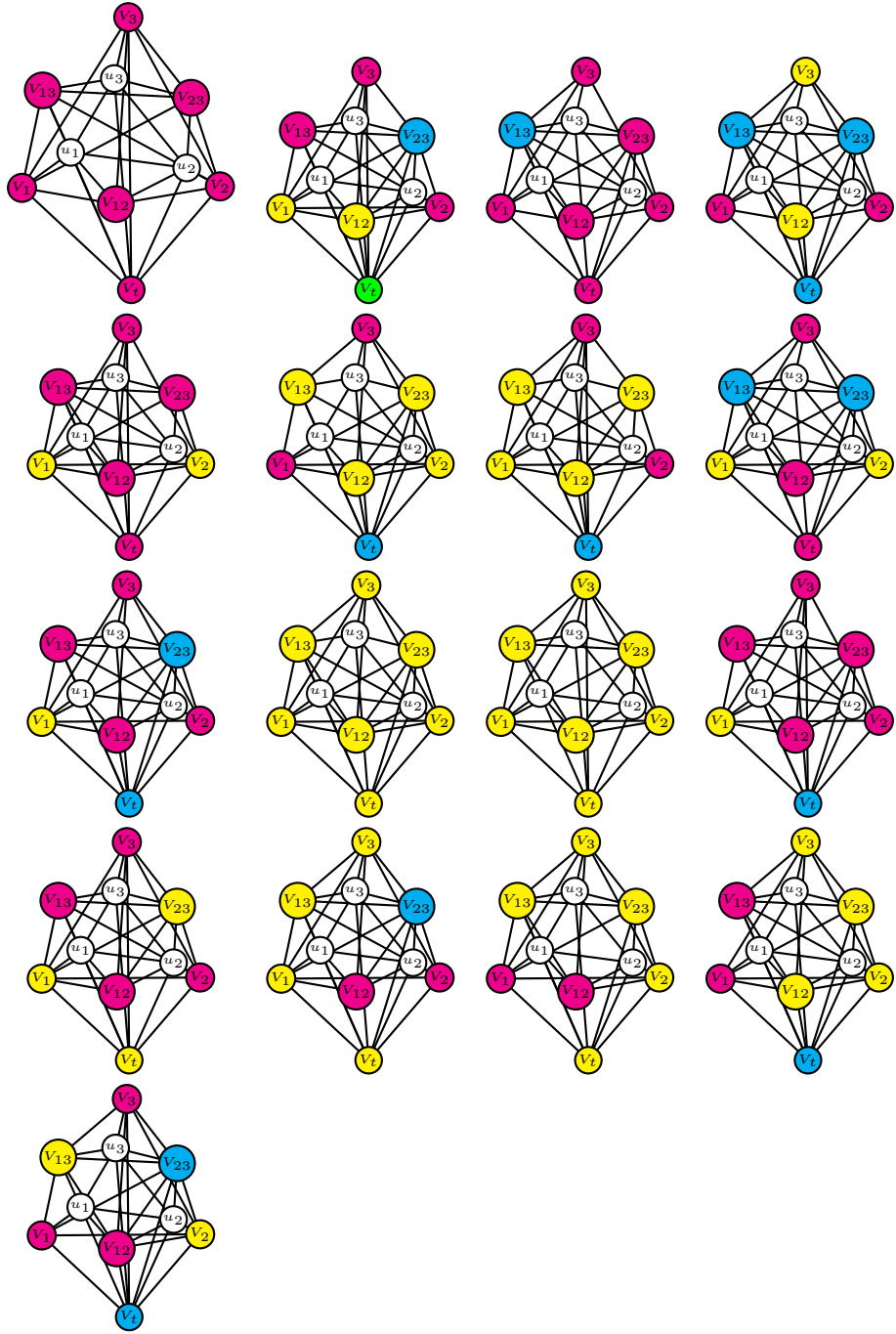
\end{document}